\documentclass[9pt]{article}

\makeatletter
\DeclareRobustCommand{\qed}{%
  \ifmmode 
  \else \leavevmode\unskip\penalty9999 \hbox{}\nobreak\hfill
  \fi
  \quad\hbox{\qedsymbol}}
\newcommand{\openbox}{\leavevmode
  \hbox to.77778em{%
  \hfil\vrule
  \vbox to.675em{\hrule width.6em\vfil\hrule}%
  \vrule\hfil}}
\newcommand{\qedsymbol}{\openbox}
\newenvironment{proof}[1][\proofname]{\par
  \normalfont
  \topsep6\p@\@plus6\p@ \trivlist
  \item[\hskip\labelsep\itshape
    #1.]\ignorespaces
}{%
  \qed\endtrivlist
}
\newcommand{\proofname}{Proof}
\makeatother

\usepackage[utf8]{inputenc}

\usepackage{bm}

\usepackage{color}
\usepackage{latexsym}
\usepackage{dsfont}
\usepackage{amssymb}
\usepackage{comment}
\usepackage{graphicx}
\usepackage{amsmath,amsfonts,amssymb,theorem,euscript,array,enumerate,amsfonts,mathrsfs}
\usepackage{hyperref}
\usepackage{appendix}
\usepackage[T1]{fontenc}
\usepackage{babel}
\numberwithin{equation}{section}
\usepackage{bbm}
\usepackage{subfigure}
\usepackage{color}
\usepackage{stmaryrd}

\usepackage[algo2e,ruled,vlined]{algorithm2e} 
 \usepackage{algorithm}
 \usepackage{algorithmicx}
\usepackage{algpseudocode}
\usepackage{ marvosym }
\usepackage{hyperref}
\usepackage{bm}
\usepackage{xcolor, soul}

\usepackage{mathtools}
\mathtoolsset{showonlyrefs}

\usepackage{comment}

\usepackage{tikz}
\usetikzlibrary{fit,matrix,chains,positioning,decorations.pathreplacing,arrows}

\usepackage{mathtools}

\usepackage{geometry}
\usepackage{algpseudocode}
\usepackage{algorithm}

\usepackage[normalem]{ulem}

\newcommand{\pr}{\text{pr}}

\newcommand{\ba}{{\bf a}}

\def \b1{\bf{1}}

\def \A{\mathbb{A}}

\def \N{\mathbb{N}}
\def \R{\mathbb{R}}

\def \E{\mathbb{E}}
\def \F{\mathbb{F}}

\def \P{\mathbb{P}}

\def \T{\mathbb{T}}

\def \bd{\mathbb{d}}

\def \bpi{\boldsymbol{\pi}} 
 
\def \bp{\boldsymbol{p}} 
\def \bpi{\boldsymbol{\pi}} 
 
\def \bF{\boldsymbol{F}} 
 
\def \bd{\boldsymbol{d}}

\def \bnu{\boldsymbol{\nu}} 
\def \bA{\boldsymbol{A}} 
\def \bu{\boldsymbol{u}} 
 
\def \balpha{\boldsymbol{\alpha}}

\def \bAc{\boldsymbol{\Ac}}

\def \mra{\mathrm{a}}

\def \d{\mathrm{d}}

\def\esssup_#1{\underset{#1}{\mathrm{ess\,sup\, }}}

\def\argmin_#1{\underset{#1}{\mathrm{argmin\, }}}
\def\argmax_#1{\underset{#1}{\mathrm{argmax\, }}}

\def\dm#1{\frac{\delta}{\delta m}}

\def \Ac{{\cal A}}
\def \Bc{{\cal B}}
\def \Cc{{\cal C}}
\def \Dc{{\cal D}}
\def \Ec{{\cal E}}
\def \Fc{{\cal F}}
\def \Gc{{\cal G}}

\def \Ic{{\cal I}}

\def \Lc{{\cal L}}
\def \Pc{{\cal P}}
\def \Rc{{\cal R}}

\def \Mc{{\cal M}}

\def \Oc{{\cal O}}
\def \Sc{{\cal S}}
\def \Tc{{\cal T}}
\def \Uc{{\cal U}}

\def \Wc{{\cal W}}
\def \Xc{{\cal X}}
\def \Yc{{\cal Y}}

\def \eps{\varepsilon}

\def\bolX{{\boldsymbol X}}

\def\boxi{{\boldsymbol \xi}}

\def\bX{{\bf X}}

\def\be{{\bf e}}
\def\bx{{\boldsymbol x}}

\def\bp{{\bf p}}

\def\bd{{\bf d}}

\def \mra{\mathrm{a}}

\def \d{\mathrm{d}}

\def\beqs{\begin{eqnarray*}}
\def\enqs{\end{eqnarray*}}
\def\beq{\begin{eqnarray}}
\def\enq{\end{eqnarray}}

\newcommand{\red}[1]{\textcolor{red}{#1}}
\newcommand{\bl}[1]{\textcolor{blue}{#1}}

\def\red#1{{\color{red}#1}}

\def\cy#1{{\color{cyan}#1}}

\newtheorem{Theorem}{Theorem}[section] 
\newtheorem{Definition}[Theorem]{Definition} 
\newtheorem{Proposition}[Theorem]{Proposition}
\newtheorem{Assumption}[Theorem]{Assumption}
\newtheorem{Lemma}[Theorem]{Lemma}

\newtheorem{Remark}[Theorem]{Remark}

\title{Non-Exchangeable Mean Field Markov Decision Processes with common noise : from Bellman equation to quantitative propagation of chaos}

\author{Samy MEKKAOUI\footnote{Ecole Polytechnique, CMAP, \sf samy.mekkaoui at polytechnique.edu; This author is supported by the S-G Chair "Risques Financiers", and the "Deep Finance and Statistics" Qube-RT Chair.}
\and 
Huy\^en PHAM\footnote{Ecole Polytechnique, CMAP, \sf huyen.pham at polytechnique.edu; This author is supported by
the BNP-PAR Chair “Futures of Quantitative Finance", the Chair “Risques Financiers", and by FiME,
Laboratoire de Finance des Marchés de l’Energie, and the “Finance and Sustainable Development” EDF -
CACIB Chair.}}

\date{}

\begin{document}

\maketitle

\begin{abstract}

We study infinite-horizon Markov Decision Processes (MDPs) with a continuum of heterogeneous agents interacting through a common noise, without assuming exchangeability. We introduce the framework of Conditional Non-Exchangeable Mean Field MDPs (CNEMF-MDPs) in both a strong formulation and a label–state formulation. We establish the equivalence between these two formulations by showing that the control problem can be lifted to a standard MDP defined on the Wasserstein space $\Pc_{\lambda}(I \times \Xc)$, where $I$ denotes the label (heterogeneity) space, $\Xc$ is the individual state space, and $\lambda$ specifies the fixed distribution of agent labels. Within this framework, we characterize the value function as the unique fixed point of an appropriate Bellman operator acting on $\Pc_{\lambda}(I \times \Xc)$.

Our second contribution is a quantitative analysis of the propagation of chaos for this non-exchangeable setting with common noise. We derive sharp finite-population bounds by comparing the Bellman operator of the finite $N$-agent MDP, defined on the high-dimensional space $\Xc^N$, with its infinite-agent counterpart. This comparison yields explicit constructions of near-optimal policies for the $N$-agent system from $\epsilon$-optimal policies of the limiting CNEMF-MDP.

\end{abstract}

\vspace{5mm}


\vspace{5mm}



\section{Introduction}

The study of large-population and complex systems is a central topic in mathematical modeling, with broad applications across modern society, including social networks, power grid infrastructures, financial markets, and lightning networks. The analysis of such systems naturally leads to the study of McKean-Vlasov (MKV) equations, in which the dynamics of a representative agent depend on the distribution of the population. These equations have been extensively investigated within the frameworks of mean field games (MFG) and mean field control (MFC). We refer the reader to the seminal lectures of P.-L. Lions \cite{lasry2007mean} and the comprehensive monographs by Bensoussan, Frehse, Yam \cite{benetal13},  Carmona and Delarue \cite{carmona_probabilistic_2018, carmona_probabilistic_2018-1} for foundational treatments of these topics.

In this work, we are motivated by the modeling of increasingly complex multi-agent networks, 
where heterogeneity and asymmetric interactions break the symmetry assumptions typically underlying MFG and MFC models. Such systems lead naturally to non-exchangeable interactions. Interest in such models has recently gained momentum, particularly through the introduction of the graphon framework in \cite{lovasz_large_2010}, which was subsequently developed in the context of heterogeneous mean field systems in  \cite{bayraktar2023graphon,bayraktar2023propagation,bayraktar_wu_2023graphon, lacker2023label, coppini2024nonlinear, crutan24}. 
These approaches have also been extended to mean field games with graphon interactions, notably in \cite{xu2025lqgmfg,alvarez2025contracting,caihua21,auretal22,berrak25}. We also refer to \cite{lauriere2025overview} for a very recent review on heterogeneous mean-field games.

Beyond graphon structures, more general non-exchangeable formulations have emerged in mean field control. 
In particular, the space $L^2_\lambda\big(I;\Pc^2(\R^d)\big)$ has been proposed to model families $(\mu^u)_{u \in I}$ of square-integrable probability measures on the agent state metric space $\R^d$, indexed by a label continuous space $I$ $=$ $[0,1]$ encoding heterogeneity, with the integrability condition: $\int_{I}\int_{\R^d} |x|^2 \mu^u(\d x) \lambda(\d u) <  \infty$, where $\lambda$ is a distribution, e.g. uniform, on the label space $I$. 
This setting, introduced in \cite{decrescenzo2024mean, decrescenzo2025linear, kharroubi2025stochastic}  allows for new analytical tools, including an adapted version of Itô’s formula \cite{decrescenzo2024mean} and a tailored notion of convexity \cite{kharroubi2025stochastic}.  
Alternatively, the paper \cite{lacker2023label} proposes a label-state formulation of graphon mean-field game, by considering the joint label/state and then working on the probability space $\Pc(I \times \Xc)$, which can be identified to an almost continuum of probability measures on $\Xc$ by disintegration. 
Recently, \cite{djete2025non} introduced a new framework for non exchangeable mean field systems where the interactions can be controlled.

In this paper, we investigate non-exchangeable mean field models in discrete time under the influence of common noise. The impact of common noise in continuous-time graphon mean field systems has been recently explored in \cite{bayetal25,de2025optimal}. On the discrete-time side, mean field control problems with symmetric (i.e., exchangeable) interactions — often referred to as mean field MDPs — have been studied in \cite{motte2022mean, motte2023quantitative, carmona2019modelfree, bau23, gu2019dynamic}. In particular, \cite{motte2022mean} and \cite{carmona2019modelfree} show that such problems can be lifted to a Markov decision process over the space of probability measures $\Pc(\Xc)$.

Our work extends this line of research to Conditional Non-Exchangeable Mean Field Markov Decision Processes (CNEMF-MDPs). We demonstrate that CNEMF-MDPs can be equivalently for\-mulated as a standard MDP on the Wasserstein space $\Pc_\lambda(I \times \Xc)$ — the space of probability measures over the joint label–state space $I \times \Xc$ with fixed first marginal $\lambda$ on $I$ (typically uniform). This formulation enables us to derive a Bellman dynamic programming equation on $\Pc_\lambda(I \times \Xc)$, which characterizes the value function of the CNEMF-MDP.

Additionally, we establish a quantitative propagation of chaos result that provides explicit convergence rates of the $N$-agent system toward the CNEMF-MDP limit. This result not only justifies the mean field approximation but also enables the construction of approximate optimal policies for the finite-agent system. This also motivates the use of numerical methods for multi-agent reinforcement learning \cite{zhang2021multi} and Mean Field-Learning  \cite{frikha2023actor,lauriere2022learning,carmona2019modelfree}. 

\paragraph{Our work and contributions.}

Inspired by the above discussion,  we consider a social planner control  problem with $N$ cooperative agents in a discrete time non exchangeable setting with common noise over an infinite horizon. The controlled state process  $\bolX^N :=( X^{i,N})_{i \in \llbracket 1, N \rrbracket}$ of the $N$-agent model is given by 
\begin{align}\label{eq:dynamics_NEMFC N agents intro}
\left\{\begin{aligned}  
       X_0^{i,N} &= x_0^i, \\
       X_{t+1}^{i,N} &=  F_N( \frac{i}{N},X_t^{i,N},\alpha_t^{i,N}, \frac{1}{N} \sum_{j=1}^{N} \delta_{(\frac{j}{N},X_t^{j,N},\alpha_t^{j,N})},\epsilon_{t+1}^{i/N}, \epsilon_{t+1}^0), \quad t \in \N.
    \end{aligned}
        \right.
\end{align}
Here, $(x_0^i)_{i \in \llbracket 1, N \rrbracket}$ are the deterministic initial states valued in a compact Polish space $\Xc$ with metric $d$, $(\epsilon_t^{i/N})_{i \in \llbracket 1 ,N \rrbracket, t \in \N^*}$ is a family of mutually i.i.d random variables on some probability space $(\Omega,\Fc,\P)$, and valued in a measurable space $E$ representing the idosyncratic noises, and $(\epsilon_t^0)_{t \in \N^*}$ is another family of i.i.d random variables valued in some measurable space $E^0$ representing the common noise, independent of the idiosyncratic noises. The control of agent $i$ denoted by $\alpha^{i,N} := (\alpha_t^{i,N})_{t \in \N}$ is a process valued in a compact Polish space $A$ with metric $d_A$ and assumed to be adapted with the filtration $\F^N :=(\Fc_t^N)_{t \in \N}$ generated by $\boldsymbol{\epsilon}^N= \big( (\epsilon_t^{i/N})_{ i \in \llbracket 1 , N \rrbracket}, \epsilon_t^0)_{t \in \N^*}$ completed with a family of mutually i.i.d uniform random variables $\boldsymbol{Z}^N = (Z_t^{i/N})_{i \in \llbracket 1, N \rrbracket, t \in \N}$ used for randomizing the controls $(\alpha^{i,N})_{i \in \llbracket 1, N \rrbracket}$. The mean-field interactions between the agents are formalized via the state function $F_N$ where the main difference between the classical mean field setting is the dependance of $F_N$ in the joint empirical measure of label-state-control. Therefore, $F_N$ is  here a $\Xc$-valued measurable function defined on $I \times \Xc \times A \times \Pc(I \times \Xc \times A) \times E \times E^0$ where $\Pc(I \times \Xc \times A)$ is the space of probability measures on $I \times \Xc \times A$. 
We now introduce the admissible set $\bAc_N$ of controls  as
\begin{align}\label{eq : def_admissible_policies_N_agents}
    \bAc_{N} := \Big\lbrace \balpha=(\alpha^{i,N}_t)_{ i\in \llbracket 1 , N \rrbracket,t \in \N} : \alpha^{i,N} \text{ is $\F^N$-adapted for each $i \in \llbracket 1, N \rrbracket$ } \Big \rbrace.
 \end{align}
The goal is then to maximise over $\balpha \in \bAc_{N}$ the following social gain functional
\begin{align}
    V_N^{\balpha}(\bx_0) := \frac{1}{N} \sum_{i=1}^{N} \E \Big[ \sum_{t \in \N} \beta^t  f_N \big(\frac{i}{N} , X_t^{i,N}, \alpha_t^{i,N} , \frac{1}{N} \sum_{j=1}^{N} \delta_{(\frac{j}{N}, X_t^{j,N}, \alpha_t^{j,N})}\big) \Big],
\end{align}
where we set $\bx_0 := (x_0^i)_{i \in \llbracket 1, N \rrbracket} \in \Xc^N$ for the initial state of the $N$ agent system. Here $\beta$ $\in$ $(0,1)$ is a discount factor, and  $f_N$ is a bounded measurable real-valued function defined on $I \times \Xc \times A \times \Pc(I \times \Xc \times A)$. 
We want to highlight that in the $N$-agent MDP, the reward and state functions depend upon $N$ as the interactions couplings may vary with the number of agents. 
The value function of this problem is defined on $\Xc^N$ as
\begin{align}\label{eq : value function N-agent MDP Intro}
    V_N(\bx_0) :=  \underset{\balpha \in \bAc_{N}}{\sup} V_N^{\balpha}(\bx_0). 
\end{align}

Let us now formulate the asymptotic mean-field formulation for the  non exchangeable mean-field control  problem when the number of agents $N$ goes to infinity.  To this end, we introduce two limiting formulations, and we will discuss their interconnections. 

\vspace{1mm}

\noindent {\it Strong formulation:}  As the dynamics of the agent state are no longer identically distributed, we hence have a continuum of agents labeled by $u \in I =[0,1]$ with the following controlled state process  dynamics $\bolX :=(X^u)_{u \in I}$  
\begin{align}\label{eq:dynamics_NEMFC_strong_formulation_intro}
    \left\{\begin{aligned}  
       X_0^u &=  \xi^u, \\
       X_{t+1}^u &=  F(u,X_t^u,\alpha_t^u, \P^0_{(X_t^v,\alpha_t^v)}(\d x , \d a ) \d v ,\epsilon_{t+1}^u, \epsilon_{t+1}^0), \quad t \in \N, \quad u \in I,
    \end{aligned}
        \right.
\end{align}
where each initial state $\xi^u$ is a $\Gc^u = \sigma(\Gamma^u)$-measurable random variable with $\Gc^u$ a $\sigma-$algebra generated by a $G$-valued random variable $\Gamma^u$ (where $G$ will denote the initial information space) independent of $\F^u := (\Fc_t^u)_{t \in \N}$ the filtration generated by $(\epsilon_t^u,\epsilon_t^0)_{t \in \N^*}$, where $(\epsilon_t^u)_{u \in I, t \in \N^*}$ is a family of mutually i.i.d random variables valued in $E$ and $(\epsilon_t^0)_{t \in \N^*}$ is another family of i.i.d random variables valued in $E^0$, completed by a family of mutually i.i.d uniform random variables $(Z^u_t)_{u \in I,t \in \N}$. $\P^0$ denotes the conditional  law knowing the common noise $\epsilon^0$. Here, $F$ denotes the state function for the asymptotic mean field control problem and stands for a $\Xc$-valued measurable function defined on $I \times \Xc \times A \times \Pc(I \times \Xc \times A) \times E \times E^0$.
We now introduce the set $\A^{OL}$ of open-loop policies as the set of sequences $\mathrm{a}$ $=$ $(\mathrm{a_t})_t$ of Borel maps 
$\mathrm{a_t}$ on $I \times G \times E^t \times (E^0)^t$ into $A$, for $t$ $\in$ $\N$,  with convention that $\mathrm{a}_0$ is defined on $I \times \Xc$, and  the admissible set of randomized open-loop controls in the strong formulation $\Ac^{S}$ as
\begin{align}\label{eq : def_admissible_controls_strong_formulation}
    \Ac^{S}  := \big \lbrace \balpha = (\alpha_t^u)_{u \in I, t \in \N} : \alpha_t^u = \mathrm{a}_t(u, \Gamma^u, (\epsilon^u_{s})_{s \leq t}, (\epsilon^0_{s})_{s \leq t } ), \;  t \in  \N, \;  \lambda \; \text{-a.e.} \;  u \in I,  
      \mbox{ for some } \; \mathrm{a} \in \A^{OL} \big \rbrace.
\end{align}
The non-exchangeable mean-field control problem consists then in maximizing  over $\Ac^S$ the following gain 
\begin{align}
   V^{\balpha}_{S}(\boxi) := \int_I \E \Big[ \sum_{t \in \N} \beta^t f \big(u,X_t^u,\alpha_t^u , \P^0_{(X_t^v,\alpha_t^v)}(\d x , \d a ) \d v \big) \Big] \d u,
\end{align}
for $\boxi$ $=$ $(\xi^u)_{u\in I}$, and where $f$ stands for the reward function as a bounded real-valued measurable function defined on $I \times \Xc \times A \times \Pc(I \times \Xc \times A)$.  
 The value function of this optimization problem is then defined 
as 
 \begin{align}\label{eq:value_function-strong-formulation}
     V_S(\boxi) &:= \;  \underset{\balpha \in \Ac^{S}}{\sup}V^{\balpha}_{S}(\boxi),
 \end{align}
and we will show that $V_S(\boxi)$ depend on $\boxi$ $=$ $(\xi^u)_u$ only through the disintegration distribution 
$\P^\lambda_{\boxi}(\d u,\d x)$ $:=$ $ \lambda(\d u)\P_{\xi^u}(\d x)$ $\in$  $\Pc_{\lambda}(I\times\Xc)$ the set of probability measures on  $I\times\Xc$ with first marginal equal to the uniform distribution $\lambda(\d u)$ $=$ $\d u$ on $I$, and where $\P_{\xi^u}$ denotes the distribution of the $\Xc$-valued r.v. $\xi^u$ under $\P$. 

\vspace{1mm}

\noindent  {\it Weak formulation :} The strong formulation introduced above can be equivalently recast in a label–state formulation, in the spirit of \cite{lacker2023label}. More precisely, we introduce a random variable $U$ taking values in $I$, whose law satisfies $\P_{U} = \lambda$. This variable is used to encode the labeling of the agents. The controlled state process $X=(X_t)_{t \in \N}$ evolves according to the following dynamics:
\begin{align}\label{eq:dynamics_NEMFC_weak_formulation_intro}
    \left\{\begin{aligned}  
       X_0 &=  \xi, \\
       X_{t+1} &=  F(U,X_t,\alpha_t, \P^0_{(U,X_t,\alpha_t)} ,\epsilon_{t+1}, \epsilon_{t+1}^0), \quad t \in \N,
    \end{aligned}
        \right.
\end{align}
where $\xi$ is a $\Gc = \sigma(\Gamma)$ measurable random variable with $\Gc$ a $\sigma$-algebra generated by a $G$-valued random variable $\Gamma$ independent of $\F = (\Fc_t)_{t \in \N}$ generated by $\boldsymbol{\epsilon}=(\epsilon_t,\epsilon_t^0)_{t \in \N^*}$ where $(\epsilon_t)_{t \in \N^*}$ is a family of mutually i.i.d random variables valued in $E$ and $(\epsilon_t^0)_{t \in \N^*}$ is another family of i.i.d random variables valued in $E^0$, completed by a family of mutually i.i.d uniform random variables $(Z_t)_{t \in \N^*}$. The admissible set of randomized open-loop controls in the weak formulation $\Ac^{W}$ is defined as 
\begin{align}\label{eq : def_admissible_controls_weak_formulation}
    \Ac^{W}  := \big \lbrace \alpha = (\alpha_t)_{ t \in \N} : \alpha_t = \mathrm{a}_t(U, \Gamma, (\epsilon_{s})_{s \leq t}, (\epsilon^0_{s})_{s \leq t}), \; t \in \N, \mbox{ for some } \; \mathrm{a} \in \A^{OL}   \big \rbrace. 
\end{align}
The non-exchangeable mean-field control problem consists in maximizing over $\Ac^W$ the following gain 
\begin{align}
    V^{\alpha}_{W}(\xi) :=  \E \Big[ \sum_{t \in \N} \beta^t f \big(U,X_t,\alpha_t , \P^0_{(U,X_t,\alpha_t)}\big) \Big], 
\end{align}  
and the associated value function is given by 
 \begin{align}\label{eq:value_function-weak-formulation}
     V_W(\xi) &:= \;  \underset{\alpha \in \Ac^{W}}{\sup} V^{\balpha}_{W}(\xi). 
\end{align}
We will show that $V_W(\xi)$ depend on $\xi$ only through the joint distribution $\P_{(U,\xi)}(\d u , \d x) = \lambda(\d u ) \P_{\xi|U=u}(\d x) \in \Pc_{\lambda}(I \times \Xc) $ where $\P_{\xi|U=u}$ denotes the law of the conditional $\Xc$-valued r.v $\xi |U=u$. Formally, in this weak formulation, the conditional law of $X|U=u$ corresponds to the law of $X^u$ for $\lambda$-a.e $u \in I$ in the strong formulation.


The control problems \eqref{eq:dynamics_NEMFC_strong_formulation_intro} - \eqref{eq:value_function-strong-formulation} and \eqref{eq:dynamics_NEMFC_weak_formulation_intro}-\eqref{eq:value_function-weak-formulation} refereed as non exchangeable mean field Markov Decision Process with common noise (CNEMF-MDP) is the natural extension of the classical mean-field setting as studied for instance in \cite{motte2022mean,motte2023quantitative}.  We emphasize that the general state equations \eqref{eq:dynamics_NEMFC N agents intro}, \eqref{eq:dynamics_NEMFC_strong_formulation_intro} and \eqref{eq:dynamics_NEMFC_weak_formulation_intro} are more general than the usual graphon frameworks which we can encounter in the current litterature as they may include graphons depending upon the state and even the control terms. 


The contributions of the present work are threefold.  

\begin{itemize}
    \item  We first show, following similar arguments to the ones presented in \cite{motte2022mean}, how the control problems \eqref{eq:dynamics_NEMFC_strong_formulation_intro}-\eqref{eq:value_function-strong-formulation} and \eqref{eq:dynamics_NEMFC_weak_formulation_intro}-\eqref{eq:value_function-weak-formulation} can be recasted as a standard mean field control  on the space $\Pc_{\lambda}(I \times \Xc)$ and how the value functions \eqref{eq:value_function-strong-formulation}-\eqref{eq:value_function-weak-formulation} can be interpreted as a fixed point of a suitable Bellman operator. Moreover, we show that for $\boxi = (\xi^u)_u$ and $\xi$ such that $\P_{\xi^u} = \P_{\xi | U=u}$ for $\lambda-\text{a.e } u \in I$, we have
    \begin{align}\label{eq : equivalence_strong_weak_formulation}
        V_S(\boxi) = V_W(\xi)=\hat{V}(\mu),
    \end{align}
    where $\mu(\d u ,\d x) = \lambda(\d u)\P_{\xi^u}(\d x)  = \lambda(\d u )\P_{\xi|U=u} \in \Pc_{\lambda}(I \times \Xc)$ and $\hat{V}$ is defined in \eqref{eq : value function canonical formulation}.
    We also show how to construct optimal randomized feedback controls for the strong and weak control problems.
    We adapted the proofs proposed in \cite{motte2022mean,motte2023quantitative} by taking advantage of their measurable coupling lemma in $\Pc(\Xc)$. The construction of optimal randomized feedback  controls in the CNEMF-MDP follows from Appendix A in \cite{motte2023quantitative}. Moreover,  the optimal randomized feedback control  for the weak and strong formulation share the same control policy $(\mathrm{a}_t)_{t \in \N}$ as a consequence of the verification result in Proposition \ref{proposition : verification theorem} together with  Proposition \ref{lemma: equivalence_between_weak_strong}.
    \item We then show the convergence of the value function $V_N$ defined in \eqref{eq : value function N-agent MDP Intro} towards the value function $\hat V$ defined indifferently in  \eqref{eq:value_function-strong-formulation} or \eqref{eq:value_function-weak-formulation} from \eqref{eq : equivalence_strong_weak_formulation} . More precisely, we show  that for all $\bx :=(x^i)_{i \in \llbracket 1, N \rrbracket} \in \Xc^N$  and $\bu = (\frac{i}{N})_{i \in \llbracket 1, N \rrbracket}$
    \begin{align}\label{eq : convergence value_functions intro}
    \big | V_N \big(\bx \big) - \hat V  \big(\mu_N^{\lambda}[\bu , \bx ]\big) \big | \underset{N \to \infty}{\longrightarrow} 0,
\end{align}
where $\mu_N^{\lambda}[\bu,\bx]$ is the lifted measure of $\frac{1}{N} \sum_{i=1}^{N} \delta_{(\frac{i}{N}, x^i)}$  in $\Pc_{\lambda}(I \times \Xc)$ with same weights over each interval, i.e.,  defined as 
\begin{align}\label{eq : mu_Ncanonicalmeasure}
    \mu_N^{\lambda}[\bu,\bx](\d u ,\d x) &:= \; \lambda(\d u)  \Big( \sum_{i=1}^{N} \mathds{1}_{[\frac{i-1}{N} , \frac{i}{N})}(u) \delta_{x^i}(\d x) \Big). 
\end{align}
Under additional assumptions on the functions $F_N,f_N,F$ and $f$, we also provide  a quantitative propagation of chaos for  \eqref{eq : convergence value_functions intro}. More precisely, we will show that \eqref{eq : convergence value_functions intro} can be rewritten for some positive constant $C$ (depending on the data of the problem)
\begin{align}\label{eq : propagation_of_chaos_intro}
    \big | V_N(\bx) - \hat V(\mu_N^{\lambda}([ \bu , \bx] \big)  \big| \leq C \big( M_N^{\gamma} + \epsilon_N^f + (\epsilon_N^F)^{\gamma} \big),
\end{align}
where $\epsilon_N^f$ and $(\epsilon_N^F)^{\gamma}$ should be respectively understood formally as convergence rates of $F^N \to F$ and $f^N \to f$. Here, $M_N$ stands for the mean rate of convergence in Wasserstein distance of the empirical measure on the product space $I \times \Xc \times A$ (see \cite{fournier2015rate}) and $\gamma \in (0,1]$ is an explicit constant equal to the Hölder regularity of the value function $\hat V$ defined in \eqref{eq : equivalence_strong_weak_formulation}. Compared to the result in \cite{motte2023quantitative}, these additional terms $\epsilon_N^f$ and $\epsilon_N^F$ are due to the potential dependence of the state and reward functions with respect to the $N$-agent. The proof of \eqref{eq : propagation_of_chaos_intro} is obtained as in \cite{motte2023quantitative} by analyzing the "proximity" as $N$ tends to $\infty$ between the Bellman operators $\Tc_N$ of the $N$-agent control problem (see \eqref{eq : Bellman_Operator_N_agents} and the Bellman operator of the CNEMF-MDP $\Tc$  (see \eqref{eq : Bellman operator def 1}). 


\item Finally, in a last part, we also show how one can construct approximate policies for the $N$-agent MDP from $\epsilon$-optimal control for the CNEMF-MDP. Precisely, we will show how to build $\Oc \big(\epsilon + M_N^{\gamma} + \epsilon_N^f + (\epsilon_N^F)^{\gamma}  \big)$ optimal policies for the $N$-agent MDP starting from an $\epsilon$-optimal randomized feedback policy for the CNEMF-MDP.
\end{itemize}

\paragraph{Outline of the paper.} The structure of the paper is as follows. Section 2 introduces the general framework of the control problem, presenting its various formulations and stating the assumptions required for the subsequent analysis. Section 3 is devoted to the main results: we establish the connection between the weak and strong mean-field formulations, as well as the link between the  $N$-agent optimization problem and its mean-field counterparts, through a quantitative propagation-of-chaos result for the value functions and the construction of $\epsilon$-optimal feedback controls. Section 4 contains the proofs of the main theorems. Finally, Appendix A collects several technical results used in the proofs of Section 4, while Appendix C gathers standard results on the Bellman operator $\Tc_N$ for the $N$-agent MDP, which are instrumental in the propagation-of-chaos analysis.

\paragraph{Notations.}
\begin{enumerate}
\item [$\bullet$] Given two measurable spaces $(E, \Ec)$ and $(F,\Fc)$ and for any measurable function $\Phi : E \to F$
and positive measure $\mu$ on $(E,\Ec)$, we denote by $\Phi \sharp \mu$ the pushforward measure on $(F,\Fc)$, i.e.,   
$\Phi {\sharp} \mu(B) = \mu\big( \Phi^{-1}(B) \big), \;  B \in \Fc$. We denote by $L^0( E;F)$ the set of measurable functions from $E$ into $F$, by $L^{\infty}(E)$ the set of real-valued bounded functions on $E$ and by $L^{\infty}_m(E) := L^{\infty}(E) \cap L^0(E;\R)$, i.e., the set of real-valued bounded measurable functions on $E$. 
\begin{enumerate}
    \item  We denote by $\Pc(E)$ the set of probability measures on $E$ equipped with the the $\sigma$-algebra of the weak convergence $\Cc(E)$ . When the measurable space $(E,\Ec)$ is endowed with probability measure $\P$, and $X : (E,\Ec) \to (F,\Fc)$ is a random variable, we will denote by $\P_X$ its probability measure on the measurable space $(F,\Fc)$, i.e., $\P_X = X \sharp \P$.
   \item  A probability kernel $\nu$ on $E \times F$ is a measurable mapping from $(E,\Ec)$ into $(\Pc(F),\Cc(F))$ . Given a probability measure $\mu$ on $(E,\Ec)$ and a probability kernel $\nu$, we define the probability measure $\mu\nu$ on $(E \times F, \Ec \otimes \Fc)$ by  
\begin{align}\label{eq : measure_from_disintegration}
    (\mu\nu)(A \times B) = \int_{A \times B} \mu(\d x_1)\nu(x_1, \d x_2), \quad \forall A \in \Ec, B \in \Fc.
\end{align}
For a product space $E_1 \times E_2$, we denote by $\pr_i$  the projection function $(x_1,x_2)$ $\in$ $E_1\times E_2$ $\mapsto$ $x_i$ $\in$ $E_i$, $i$ $=$ $1,2$.   
For a product space $E_1 \times E_2 \times E_3$, we denote by $\pr_{ij}$ the projection function $(x_1,x_2,x_3) \in E_1 \times E_2 \times E_3 \mapsto (x_i,x_j) \in E_i \times  E_j$, $i,j$ $=$ $1,2,3$, $i\neq j$. 
  \item  Given $I :=[0,1]$ endowed with its Borel $\sigma$-algebra $\Bc(I)$ and $\lambda$ its uniform measure , we introduce the Polish space
   $ \Pc_{\lambda}(I \times \Xc) := \big \lbrace \mu \in \Pc(I \times \Xc) : \text{$\text{pr}_{1} \sharp \mu$ = $\lambda $} \big \rbrace$.
  \item  
When $(\Xc,d)$ is a metric space, the set $\Pc(\Xc)$ of probability measures on $\Xc$ is equipped with the Wasserstein distance $ \Wc(\mu,\mu') = \underset{\pi \in \Pi(\mu,\mu')}{\inf} \int_{\Xc \times \Xc} \d(x,x') \pi(\d x,\d x')$ where $\Pi(\mu,\mu')$ is the set of probability measures $\pi$ on $\Xc \times \Xc$ such that $\pr_{1} \sharp\pi = \mu$ and $\pr_{2} \sharp\pi= \mu'$. We recall the Kantorovitch duality result
\begin{align}\label{eq : Kantorovitch duality}
    \Wc(\mu,\mu') =  \underset{\phi \in \text{Lip}_{1}}{\sup} \int_{\Xc} \phi(x) (\mu - \mu')(\d x),
\end{align}
where $\text{Lip}_{1}$ is the set of all Lispchitz functions on $\Xc$ with Lipschitz constant bounded by 1. 
\end{enumerate}

\item [$\bullet$] For two metric spaces $(\Xc,d)$ and $(A,d_A)$, we endow the product space $\Xc \times A$ with the metric 
\begin{align}
    \bd\big( (x,a),(x',a') \big) &:= \;  d(x,x') + d_A(a,a'),\quad (x,x') \in \Xc^2, (a,a') \in A^2. 
\end{align}
\begin{enumerate}
\item Similarly, we endow the space $\Xc^N$ and $A^N$ with respectively the metrics
\begin{align}
\begin{cases}
    \bd_N(\bx,\bx') &:= \frac{1}{N} \sum_{i=1}^{N} \d (x^i, x'^i), \quad \bx=(x^i)_{i \in \llbracket 1, N \rrbracket}, \bx' = (x'^i)_{i \in \llbracket 1, N \rrbracket} \in \Xc^N, \\
    \bd_{A,N}(\ba,\ba') &:=  \frac{1}{N} \sum_{i=1}^{N} \d_A(a^i,a'^i), \quad \ba = (a^i)_{i \in \llbracket 1, N \rrbracket}, \ba'=(a'^i)_{i \in \llbracket 1, N \rrbracket} \in A^N.
\end{cases}
\end{align}
Finally, we endow the space $(\Xc \times A)^N$ with the metric
\begin{align}
    \bd_N\big((\bx,\ba),(\bx',\ba')\big) := \frac{1}{N} \sum_{i=1}^{N}  \bd((x^i,a^i),(x'^i,a'^i)\big) = \bd_N(\bx,\bx') + \bd_{A,N}(\ba, \ba').
\end{align}

\item Given $\bx = (x^i)_{i \in \llbracket 1, N \rrbracket} \in \Xc^N$, $\ba = (a^i)_{i \in \llbracket 1, N \rrbracket} \in A^N$, $\bu := (\frac{i}{N})_{i \in \llbracket 1, N \rrbracket}$, $I =[0,1]$ and ${I_i^N} :=[\frac{i-1}{N}, \frac{i}{N})$ for any $i \in \llbracket 1, N \rrbracket$, we introduce
\begin{align}\label{eq : Empirical Measures def}
\begin{cases}
     &\mu_N \big[ \bu , \bx \big]:= \frac{1}{N} \sum_{i=1}^{N} \delta_{(\frac{i}{N}, x^i)},  
     \quad \mu_N \big[\bu, \bx, \ba \big] := \frac{1}{N} \sum_{i=1}^{N}  \delta_{(\frac{i}{N}, x^i,a^i )},  
     \\
    &\mu_N^{\lambda} \big[\bu , \bx \big]:= \Big(\sum_{i=1}^{N} \mathds{1}_{I^N_i}( u) \delta_{x^i}(\d x)\Big) \d u, \quad \mu_N^{\lambda} \big[\bu , \bx , \ba  \big]:= \Big(\sum_{i=1}^{N} \mathds{1}_{I^N_i}( u) \delta_{(x^i,a^i)}(\d x,\d a)\Big) \d u 
\end{cases}
\end{align}
and we recall that 
\begin{align}\label{eq : inequality wasserstein empirical measures}
    \Wc \Big( \mu_N \big[\bx,\ba \big], \mu_N \big[\bx',\ba'\big] \Big) &\leq \;  \bd_N\big((\bx,\ba),(\bx',\ba')\big).
\end{align}
\item We denote by $\Delta_{I \times \Xc} (\text{resp } \Delta_{I \times \Xc \times A})$ the diameter of the compact metric space $I \times \Xc$ (resp $I \times \Xc \times A)$ and  we define
\begin{align}\label{eq : M_N definition}
    M_N := \underset{\mu \in \Pc(I \times \Xc \times A)}{\sup} \E \big[\Wc(\mu_N, \mu) \big],
\end{align}
where $\mu_N := \frac{1}{N} \sum_{i=1}^{N} \delta_{Y_i}$, with $(Y_i)_{1 \leq i \leq N}$ are i.i.d random variables with law $\mu$, and we recall from \cite{fournier2015rate} that $M_N \underset{N \to \infty}{\to} 0$  with non asymptotic bounds depending on the dimension of the space. We recall that $\frac{1}{N} = o(M_N)$ regardless of the dimension.
\end{enumerate}
\item [$\bullet$] In the sequel, $C$ will denote a generic constant depending of the parameters of the model which can vary from line to line.

\end{enumerate}

\section{Problem formulation and assumptions}

\subsection{Global assumptions for the Bellman fixed point equation}\label{subsec : global_assumptions}

Let $\Xc$ and $A$ be two compact Polish spaces, representing the state and action spaces, equiped with their metrics $d$ and $d_A$. We also consider the interval $I = [0,1]$ representing the continuum of heterogeneous agents endowed with its Borel $\sigma$-algebra and with the Lebesgue measure $\lambda$ which will we denote interchangeably in the sequel as $\lambda(\d u) := \d u$. We denote by $G,E,E^0$  three measurable spaces representing respectively the initial information, idiosyncratic noise, and common noise spaces. 

We now fix a probability space $(\Omega,\Fc,\P)$ on which all the following random variables will be defined, and we fix functions
\begin{align}
    F : I \times \Xc \times A \times \Pc(I \times \Xc \times A) \times E \times E^0 \to \Xc, \quad  \quad     f : I \times \Xc \times A \times \Pc(I \times \Xc \times A) \to \R,
\end{align}
assumed to satisfy the following assumptions.
\begin{Assumption}\label{assumptions: F and f}
The functions $F$ and $f$ are Borel measurable. There exist constants $L_{F}, L_f\geq 0$ such that
\begin{itemize} 
\item[(i)] $$ \E \Big[d \big(F(u,x,a,\mu,\epsilon, e^0),F(u',x',a,\mu' , \epsilon,e^0) \big) \Big] \leq L_F \big ( |u-u'| + \d(x,x') + \Wc(\mu,\mu') \big)$$
 \item[(ii)] 
 $$
\big | f(u,x,a,\mu)  - f(u,x',a,\mu')  \big | \leq  L_f \big( |u-u'| +  \d(x,x') + \Wc(\mu,\mu') \big) 
 $$
\end{itemize} 
for every  $u,u' \in I$, $x,x'\in\Xc$, $a \in A$, $\mu,\mu' \in \Pc(I \times \Xc \times A)$ and $e^0 \in E^0$, and 
where $\epsilon$ is a $E$-valued random variable on $(\Omega,\Fc,\P)$. 
\end{Assumption}

\subsection{The weak/label-state formulation}\label{subsec : label-state formulation}

We consider, in the weak formulation of the non-exchangeable mean field controlled system, the following state process:
\begin{align}\label{eq:dynamics_NEMFC_weak_formulation}
    \left\{\begin{aligned}  
       X_0 &= \;  \xi, \\
       X_{t+1} &= \;   F(U,X_t,\alpha_t, \P^0_{(U,X_t,\alpha_t)} ,\epsilon_{t+1}, \epsilon_{t+1}^0), \quad t \in \N,
    \end{aligned}
        \right.
\end{align}
where $U$ is a random variable on $I$ distributed according to $\lambda$, $\epsilon$ $=$ $(\epsilon_t)_{t \in \N^*}$ is a family of mutually i.i.d. random variables valued in $E$ with law given by $\lambda_{\epsilon}$, and $\epsilon^0$ $=$ $(\epsilon_t^0)_{t \in \N^*}$ is another family of i.i.d. random variables valued in $E^0$. 
$U$, $\eps$ and $\eps^0$ are mutually independent, and   $\P^0$ denotes the conditional law knowing the common noise $\epsilon^0$.

The initial condition is given by a $\Xc$-valued random variable $\xi$ which is $\sigma(U) \vee \Gc$-measurable where  $\Gc$ is the $\sigma$-algebra  generated by a random variable $\Gamma$ representing the initial information, independent of the filtration generated by $\boldsymbol{\epsilon} = (\epsilon_t,\epsilon_t^0)_{t \in \N^{\star}}$. 
We will assume that there exists a uniform random variable $Z \sim \Uc([0,1])$ which is $\Gc$-measurable and independent of $\xi$. We refer to this assumption as the \textbf{randomization hypothesis in the weak formulation}.

The admissible control  is given by open-loop controls $\alpha=(\alpha_t)_{t \in \N}$ which are adapted to the filtration $\F =(\Fc_t)_{t \in \N}$ generated by $(U,\Gamma,\epsilon, \epsilon^0)$, i.e., $\Fc_t = \sigma\big( U, \Gamma, (\epsilon_s)_{s \leq t}, (\epsilon_s^0)_{s \leq t} \big)$. Therefore, one can represent an admissible control process $\alpha$ through a collection of Borel maps (called open-loop policies) $\mra_t : I \times G \times E^t \times (E^0)^t \to A$, $t$ $\in$ $\N$ as  
\begin{align}
    \alpha_t = \mra_t(U,\Gamma,(\epsilon_s)_{s \leq t},(\epsilon_s^0)_{s \leq t}), \quad t \in \N, \quad \P-\text{a.s},
\end{align}
with convention that $\mathrm{a}_0$ is defined on $I \times G$. Such a control $\alpha=(\alpha_t)_{t \in \N}$ is called a randomized open-loop control and we denote by $\Ac^{W}$ this set of controls which has been formally defined in \eqref{eq : def_admissible_controls_weak_formulation}.
For $\alpha \in \Ac^{W}$ and an initial condition $\xi$, we write the expected gain $V_W^{\alpha}$ as
\begin{align}\label{eq : cost functional V weak formulation}
      V_W^{\alpha}(\xi) &:= \;  \E \Big[ \sum_{t \in \N}  \beta^t f(U,X_t,\alpha_t, \P^0_{(U,X_t,\alpha_t)}) \Big].
\end{align}
The value function of the conditional non exchangeable mean field control Markov decision processes (CNEMF-MDP) is then defined in the weak formulation by
\begin{align}\label{eq : value function weak formulation}
    V_W(\xi ) &:= \;  \underset{\alpha \in \Ac^W}{\sup} V^{\alpha}_{W}(\xi).
\end{align}

\begin{Remark}
\normalfont
The proposed model \eqref{eq:dynamics_NEMFC_weak_formulation} is more general than those studied in the literature. However, this level of generality requires a regularity assumption on the label $u \in I$. Such an assumption is also standard when deriving propagation of chaos results (see \cite{cao2025probabilistic, coppini2024nonlinear} for a  blockwise Lipschitz continuity assumption). Formally, given a graphon interaction structure $G \in L^2(I\times I)$, one could consider
\begin{align}\label{eq : graphon_based_models}
\begin{cases}
    X_0 = \xi, \\
    X_{t+1} = \hat{F}\!\Big( X_t, \alpha_t,
    \dfrac{\int_I G(U,v)\, \P^0_{(U,X_t,\alpha_t)}(\d v,\d x, \d a)}
    {\int_I G(U,v)\, \d v},
    \epsilon_{t+1}, \epsilon_{t+1}^0 \Big),
\end{cases}
\end{align}
where $\hat{F}$ is a measurable mapping from 
$ \Xc \times A \times \Pc(\Xc \times A) \times E \times E^0$ into $\Xc$, and we assume that
\[
\int_I G(u,v)\, \d v > 0 \quad \text{for $\lambda-$a.e } u \in I,
\]
that is, there are no isolated particles. The model \eqref{eq : graphon_based_models} is a particular instance of the class of models encompassed by our present setting.
\end{Remark}

\subsection{The strong formulation}

We now introduce the strong formulation  for the non exchangeable mean field controlled system \footnote{In the following, given a transition kernel $\nu$ on $I \times (\Xc \times A)$ and for the Lebesgue measure $\lambda$ on $I$, we write  indifferently the measure product $\lambda \nu$ as $\lambda(\d u ) \nu(v, \d x , \d a)$ or $\nu(v , \d x , \d a) \d v$.}:
\begin{align}\label{eq:dynamics_NEMFC}
    \left\{\begin{aligned}  
       X_0^u &= \;  \xi^u, \\
       X_{t+1}^u &= \;   F(u,X_t^u,\alpha_t^u, \P^0_{(X_t^v,\alpha_t^v)}(\d x , \d a ) \d v   ,\epsilon_{t+1}^u, \epsilon_{t+1}^0), \quad t \in \N, \quad u \in I, 
    \end{aligned}
        \right.
\end{align}
where $(\epsilon_t^u)_{u \in I, t \in \N^*}$ is a family of mutually i.i.d. random variables valued in $E$ with law given by $\lambda_{\epsilon}$. We denote for  $\lambda- \text{a.e } u \in I$, the filtration $\F^u = (\Fc_t^u)_{t \in \N}$ generated by $(\Gamma^u,\epsilon^u, \epsilon^0)$, i.e, $\Fc_t^u = \sigma \big( \Gamma^u, (\epsilon^u_s)_{s \leq t}, (\epsilon_s^0)_{s \leq t} \big)$, where $\Gamma^u$ represents the initial information available for agent $u$, and we denote by $\Gc^u = \sigma(\Gamma^u)$ its generated $\sigma-$algebra. We consider that $\Gc^u$ is independant of the filtration generated by $(\epsilon_t^u,\epsilon_t^0)_{t \in \N^{\star}}$ and $(\Gamma^u)_{u \in I}$ is a family of mutually independent random variables.
For $\lambda-\text{a.e } u \in I$, the initial condition is given by a $\Xc$-valued random variables $\xi^u$ which is $\Gc^u$ measurable. Moreover, we also assume that there exists an independent family $\lbrace Z^u : u \in I \big \rbrace$ of uniform random variable such that  for $\lambda-\text{a.e }$ $u \in I$, $Z^u$ is $\Gc^u$-measurable and independent of $\xi^u$. We  refer to this assumption as the \textbf{randomization hypothesis in the strong formulation}.

\begin{Remark}\label{rmk : randomization_uniform_decimals}
\normalfont
    Following Lemma 2.2.1 in \cite{kallenberg2002foundations} and for $\lambda-\text{a.e } u \in I$ , one can extract from $Z^u$ an i.i.d sequence of uniform random variables $(Z_t^u)_{t \in \N}$ which are $\Gc^u$-measurable and independant of $\xi^u$. This sequence will be used for the randomization of the subsequent actions.
\end{Remark}

As discussed in \cite{sun06}, it is known that we cannot construct a joint measurable version of the map $(u,\omega) \mapsto \epsilon^u(\omega)$ due to the fact that $\lbrace \epsilon^u : u \in I \rbrace$ is an uncountable collection of i.i.d. random variables. However, from a control perspective, only the measurability at the level of the law is required to define the control problem. Therefore, the following assumptions will ensure the well-posedness of the system.

\begin{Assumption}\textnormal{(Measurability in law of the initial information).}\label{assumption : measurability initial information}
\begin{itemize}
    \item [(1)] The function $        I \ni u \mapsto \P_{\Gamma^u} \in \Pc(G) \text{ is measurable}$.
    \item [(2)] We have   $ \P_{\Gamma^u} = \P_{\Gamma |U=u}, \quad \text{ for $\lambda-\text{a.e } u \in I$ }$ where $(\P_{\Gamma | U=u})_{u \in I}$ denotes a measurable version of the conditional law of $\Gamma$ given $U$ from Section \ref{subsec : label-state formulation}.
\end{itemize} 
\end{Assumption}

\begin{Definition}\textnormal{(Admissible conditions).}\label{assumption : Admissible conditions for the strong formulation}
\begin{itemize}
    \item [(1)] We say that $\boldsymbol{\xi} =(\xi^u)_{u \in I}$ is an admissible initial condition if there exists a Borel measurable function $\xi_0 :  I \times G \to \Xc$ such that we have for $\lambda-a.e$ $u \in I$
    \begin{align}
        \xi^u = \xi_0(u,\Gamma^u), \quad \P-\text{a.s}.
    \end{align}
    \item [(2)] We say that $\balpha = (\alpha^u)_{u \in I}$ is an admissible control if there exists a collection of Borel maps $(\mra_t) : I \times G \times E^t \times (E^0)^t \to A$ such that for $\lambda-a.e$ $u \in I$
    \begin{align}
        \alpha_t^u = \mathrm{a}_t \big(u , \Gamma^u , (\epsilon^u_{s})_{s \leq t}, (\epsilon^0_s)_{s \leq t} \big), \quad t \in \N, \quad \P-\text{a.s},
    \end{align}
    with convention that $\mathrm{a}_0$ is defined on $I \times G$.
\end{itemize}
Such a collection of controls $\balpha=(\alpha^u)_{u \in I}$ is called randomized-open loop control in the strong formulation and we denote by $\Ac^{S}$ this set of controls, which has been formally defined in \eqref{eq : def_admissible_controls_strong_formulation}
\end{Definition}

\begin{Remark}
\normalfont
    In the standard mean-field setting, the initial state is a measurable function of $\Gamma$ but here we need to assume the Borel-measurability with respect to the label $u \in I$ so that by Assumption \ref{assumption : measurability initial information}, the mapping $I \ni u \mapsto \P_{\xi^u} \in \Pc(\Xc)$ is measurable. 
\end{Remark}

We now give the following Lemma which ensures the well posedness of the strong formulation for the controlled system \eqref{eq:dynamics_NEMFC}.

\begin{Lemma}\label{lemma : measurability of mapping}
    Under Assumptions \ref{assumptions: F and f} and \ref{assumption : measurability initial information}, for any admissible initial condition $\boldsymbol{\xi}$ and any admissible control $\balpha \in \Ac^S$, the mapping $I \ni u \mapsto \P_{(X_t^u,\alpha_t^u,(\epsilon_s^0)_{s \leq t} )} \in \Pc(\Xc \times A \times (E^0)^t)$ is measurable for any $t \in \N$.
\end{Lemma}
\begin{proof}
    We will prove  this result by induction on $t \in \N$. Starting from $t=0$ and since $\P_{(\xi^u,\alpha_0^u)} = \big(\xi_0(u,\cdot), \mathrm{a}_0(u,\cdot)\big) \sharp \P_{\Gamma^u}$, the result follows from Assumption \ref{assumption : measurability initial information} and Definition \ref{assumption : Admissible conditions for the strong formulation}. Suppose that it holds true for a given $t \in \N^*$.

    \noindent {\it Step 1 : } We show that for any $t \in \N$, there exists a Borel measurable function $x_t : I  \times G \times E^{t} \times (E^0)^{t}  \to \Xc$ such that for a.e $u \in I$
    \begin{align}\label{eq : representation for X}
        X_t^u = x_t(u,\Gamma^u, (\epsilon^u_s)_{s \leq t}, (\epsilon^0_s)_{s \leq t}) \quad \P -\text{a.s}.
    \end{align}
    It clearly holds true at time $t=0$. Assume that it holds true at time $t \in \N^*$. By induction hypothesis, it is clear that the mapping $u \mapsto \P_{(X_t^u, \alpha_t^u, (\epsilon^0_s)_{s \leq t})}$ is measurable. Indeed, we notice that
    \begin{align}\label{eq : measurability_joint_version}
        \P_{(X_t^u,\alpha_t^u,(\epsilon_s^0)_{s \leq t})}&=\big(x_t(u,\cdot,\cdot,\cdot),\mathrm{a}_t(u,\cdot,\cdot,\cdot), \cdot \big) \sharp \P_{(\Gamma^u,(\epsilon_s^u)_{s \leq t}, (\epsilon_s^0)_{s \leq t})} \\ &=\big(x_t(u,\cdot,\cdot,\cdot),\mathrm{a}_t(u,\cdot,\cdot,\cdot), \cdot \big) \sharp (\P_{\Gamma^u} \otimes (\lambda_{\epsilon})^{\otimes (t-1)} \otimes \big(\P_{\epsilon^0_1}\big)^{\otimes (t-1)} \big),
    \end{align}
    where we used the induction hypothesis for the representation of $X_t^u$, the fact that $\balpha \in \Ac^S$ and the independence between $\Gamma^u, (\epsilon_s^u)_{s \leq t}$ and $(\epsilon_s^0)_{s \leq t}$.
    Moreover, from the definition of the conditional law, the measurability of the map $u \mapsto \P_{(X_t^u, \alpha_t^u, (\epsilon^0_s)_{s \leq t})}$, and due to the disintegration theorem on Polish spaces, there exists for every $t \in \N$ a measurable mapping $\phi_t : I \times  (E^0)^{t} \to \Pc(\Xc \times A)$ such that $\P^0_{(X_t^u,\alpha_t^u)}(\d x, \d a) = \phi_t(u,(\epsilon^0_{s})_{s \leq t})(\d x, \d a) \quad \d u \otimes \d \P \text{ a.e}$.
    Hence, one can define the measure-valued map $\tilde{\phi}$ as 
    \begin{align}\label{eq : kernel_measure}
       (E^0)^t \ni y \mapsto  \tilde{\phi}_t(y)=  \phi_t(u,y,\d x , \d a) \d u  \in \Pc(I \times \Xc \times A).
    \end{align}
     The measurability of $\tilde{\phi}_t$
follows from  the measurability of the compositions $y \to \big(\phi_t(u,y), \lambda \big) \to   \phi_t(u,y,\d x,  \d a)\d u  $. Therefore we have from the dynamics in \eqref{eq:dynamics_NEMFC}, for $\lambda -a.e$ $u \in I$
    \begin{align}
        X_{t+1}^u = F(u, x_t(u,\Gamma^u, (\epsilon^u_{s})_{s \leq t} , (\epsilon^0_s)_{s \leq t}), \mathrm{a}_t(u,\Gamma^u, (\epsilon^u_s)_{s \leq t} , (\epsilon^0)_{s \leq  t}),  \tilde{\phi}_t\big((\epsilon^0_s)_{s \leq t}),\epsilon_{t+1}^u, \epsilon_{t+1}^0), \quad \P-\text{a.s}.
    \end{align}
    We now define the measurable map $x_{t+1}$ on $I \times G \times E^{t+1} \times (E^0)^{t+1}$ as 
    \small
    \begin{align}\label{eq : measurable mapping xt+1}
        x_{t+1}(u, \gamma,(e_s)_{s \leq t+1}, (e^0_s)_{s \leq t+1})  = F(u,x_t(u,\gamma, (e_s)_{s \leq t}, (e^0_s)_{s \leq t}), \mathrm{a}_t(u,\gamma,(e_s)_{s \leq t}, (e^0_s)_{s \leq t}), \tilde{\phi}_t((\epsilon_s^0)_{s \leq t}), e_{t+1}, e^0_{t+1})
    \end{align}
    \normalsize
    where the measurability follows from the Borel measurability of $F$ by Assumption \ref{assumptions: F and f}, the induction hypothesis and the measurability of $\mathrm{a}_t$.

\noindent \textit{Step 2 :} From the representation of $X_{t+1}^u$ in \eqref{eq : representation for X} with the measurable mapping $x_{t+1}$ defined in \eqref{eq : measurable mapping xt+1}, it is again clear following the same computation as in \eqref{eq : measurability_joint_version} that the mapping $u \mapsto \P_{(X_{t+1}^u,\alpha_{t+1}^u, (\epsilon_s^0)_{s \leq t+1})}$ is measurable which is enough to the proof.
\end{proof}

\vspace{2mm}

From Lemma \ref{lemma : measurability of mapping}, we have the measurability of the mapping $u \mapsto \P_{(X_t^u,\alpha_t^u, (\epsilon^0_s)_{s \leq t})}$ for any $t \in \N$ which ensures the well posedness of the controlled system \eqref{eq:dynamics_NEMFC}. 
Given an admissible condition $\boldsymbol{\xi}$ and an admissible control $\balpha \in \Ac^S$, we can then define the expected gain in the strong formulation as
\begin{align}\label{eq : cost functional V form 1}
    V_S^{\balpha}(\boldsymbol{\xi}) &:= \;  \int_I \E \Big[  \sum_{t \in \N} \beta^t f(u, X_t^u, \alpha_t^u, \P^0_{(X_t^v,\alpha_t^v)}(\d x , \d a )\d v ) \Big] \d u.
\end{align}
Moreover, it is worth mentioning that \eqref{eq : cost functional V form 1} can be rewritten as 
  \begin{align}\label{eq : equivalent_formulation_expected_gain}
        V_S^{\balpha}(\boldsymbol{\xi}) &= \;  \E \Big[ \sum_{t \in \N}  \beta^t  \int_I \E^0  \big[ f \big(u,X_t^u,\alpha_t^u, \P^0_{(X_t^v,\alpha_t^v)}(\d x, \d a ) \d v \big)  \big] \d u  \Big],
    \end{align}
from Fubini's theorem and the law of iterated conditional expectations since the mapping
\begin{align}
    I \times \Omega \ni (u,\omega) \to \E^0 \big[ f(u,X_t^u,\alpha_t^u,\P^0_{(X_t^v,\alpha_t^v)}(\d x , \d a ) \d v) \big](\omega),
\end{align}
is jointly measurable. Indeed, the measurability follows from the measurability of the mapping $u \mapsto \P_{(X_t^u,\alpha_t^u,(\epsilon^0_s)_{s \leq t})}$ in Lemma \ref{lemma : measurability of mapping} and the Borel measurability of $f$ after noticing that $\P-\text{a.s}$ (where $\tilde{\phi}_t$ has been defined in \eqref{eq : kernel_measure})
    \begin{align}
        \E^0 \big[ f(u,X_t^u,\alpha_t^u,\P^0_{(X_t^v,\alpha_t^v)}(\d x , \d a ) \d v) \big](\omega) = \int_{\Xc \times A} f \big(u,x,a, \tilde{\phi}_t\big((\epsilon^0_{s})_{s \leq t}(\omega) \big) \P_{(X_t^u, \alpha_t^u) | (\epsilon^0_s)_{s \leq t} = (\epsilon^0_{s})_{s \leq t}(\omega)}(\d x, \d a).
    \end{align}
The value function of the conditional non exchangeable mean field control Markov decision processes CNEMF-MDP is then defined by
\begin{align}\label{eq : value function}
    V_S(\boldsymbol{\xi} ) &:= \;  \underset{\balpha \in \Ac^S}{\sup} V_S^{\balpha}(\boldsymbol{\xi}).
\end{align}

We end this subsection by showing the connection between the strong and weak formulations.

\begin{Proposition}\label{lemma: equivalence_between_weak_strong}

    Under Assumptions  \ref{assumptions: F and f} and \ref{assumption : measurability initial information}, let  $\boldsymbol{\xi} :=(\xi^u)_{u \in I}$ be a collection of $\Xc$-valued random variables admissible in the sense of  Definition \ref{assumption : Admissible conditions for the strong formulation} for a given Borel map $\xi_0$ and let $\xi = \xi_0(U,\Gamma)$.  Then given a collection of open-loop policies  $(\mra_t)_{t \in \N}$ and defining the controls $\balpha=(\alpha^u)_{u \in I} \in \Ac^S$ and $\alpha \in \Ac^W$, respectively as 
    \begin{align}\label{eq : representation alpha}
    \begin{cases}
        \alpha_t^u &= \mra_t(u, \Gamma^u, (\epsilon_s^u)_{s \leq t }, (\epsilon_s^0)_{s \leq t}), \\
        \alpha_t &= \mra_t(U,\Gamma, (\epsilon_s)_{s \leq t},(\epsilon_s^0)_{s \leq t}), 
    \end{cases}
    \end{align}
    we have $V_S^{\balpha}(\boldsymbol{\xi} ) = V_W^{\alpha}(\xi)$.
\end{Proposition}
\begin{proof}
    Taking the conditional expectation with respect to $(\epsilon^0_t)_{t \in \N}$ in \eqref{eq : cost functional V weak formulation}, and  recalling that the $\Pc_{\lambda}(I \times \Xc \times A)$ valued-processes $\big(\P^0_{(U,X_t,\alpha_t)}\big)_{t \in \N}$  and $\P^0_{(X_t^v,\alpha_t^v)}(\d x, \d a) \d v $ are $\F^0$-adapted (see for instance Proposition A.1 in \cite{motte2022mean}), we see that $V^{\balpha}_S(\boldsymbol{\xi})$ and $V_W^{\alpha}(\xi)$ defined respectively in  \eqref{eq : equivalent_formulation_expected_gain} and in \eqref{eq : cost functional V weak formulation} can be rewritten respectively as
    \begin{align}
    \begin{cases}
         V_S^{\balpha}(\boldsymbol{\xi}) &= \E \Big[ \sum_{t \in \N} \beta^t  \int_{I \times \Xc \times A} f(u,x,a, \P^0_{(X_t^v,\alpha_t^v)}(\d x,\d a)\d v) \P^0_{(X_t^u,\alpha_t^u)}(\d x , \d a) \d u  \Big], \\
         V_W^{\alpha}(\xi) &= \E \Big[ \sum_{t \in \N} \beta^t \int_{I \times \Xc \times A} f(u,x,a,\P^0_{(U,X_t,\alpha_t)})\P^0_{(U,X_t,\alpha_t)}(\d u,\d x, \d a) \Big].
    \end{cases}
    \end{align}
    The rest of the proof consists then in showing that $\P^0_{(U,X_t,\alpha_t)} = \P^0_{(X_t^v,\alpha_t^v)} \d v, \quad t \in \N, \quad  \P-\text{a.s}$.  Since $U$ is independent with respect to $(\epsilon_t^0)_{t \in \N^*}$, by standard disintegration theorem, it is enough to verify that for $\lambda-\text{a.e}$ $u \in I$, we have $\P^0_{(X_t,\alpha_t) | U=u} = \P^0_{(X_t^u,\alpha_t^u)}$ for any $t \in \N$. This is clear at $t=0$ since we have for $\lambda-\text{a.e}$ $u \in I$
    \begin{align}
        \P_{(\xi^u,\alpha_0^u)} = \big( \xi_0(u,\cdot), \mathrm{a}_0(u,\cdot)\big) \sharp \P_{\Gamma^u} = \big( \xi_0(u,\cdot), \mathrm{a}_0(u,\cdot)\big) \sharp \P_{\Gamma |U=u} = \P_{(\xi_0,\mathrm{a_0})|U=u},
    \end{align}
    where the second equality follows from Assumption \ref{assumption : measurability initial information}. 
    We now prove the result by induction on $t \in \N^{\star}$. Suppose it holds true for any $t \in \N^{\star}$. Then, from the definition of $\tilde{\phi}_t$ in \eqref{eq : kernel_measure} and the induction hypothesis, it is easy to verify that $X_{t+1} = x_{t+1}(U,\Gamma,(\epsilon_s)_{s \leq t+1}, (\epsilon_s^0)_{s \leq t+1})$ $\P-\text{a.s}$ for the same Borel map $x_{t+1}$ as defined in  \eqref{eq : representation for X}. Therefore, using \eqref{eq : representation alpha},we have for $\lambda-\text{a.e } u \in I$
    \begin{align}
    \begin{cases}
     \P^0_{(X_{t+1}^u,\alpha_{t+1}^u)} &= \big(x_{t+1}(u,\cdot,\cdot, (\epsilon_s^0)_{s \leq t+1}),\mra_{t+1}(u,\cdot,\cdot,(\epsilon_s^0)_{s \leq t+1}) \sharp \big( \P_{\Gamma^u} \otimes {(\lambda_{\epsilon})^{\otimes t})}  \big), \\
     \P^0_{(X_{t+1},\alpha_{t+1})| U=u} & =\big(x_{t+1}(u,\cdot, \cdot, (\epsilon_s^0)_{s \leq t+1}), \mra_{t+1}(u,\cdot,\cdot,(\epsilon_s^0)_{s \leq t+1})\sharp \big( \P_{(\Gamma|U=u} \otimes  (\lambda_{\epsilon})^{\otimes t})),
    \end{cases}
    \end{align}
    where we used in the first equality independence between $\Gamma^u$ and $(\epsilon_s^u)_{s \in \N^{\star}}$ and  in the second equality  independence between $(U,\Gamma)$ and $(\epsilon_s)_{s \in \N^{\star}}$. From Assumption \eqref{assumption : measurability initial information}, we conclude the proof.
\end{proof}

\subsection{The $N$-agent formulation}

Let  $N \in \N^*$ representing the number of interacting particles/agents. We consider the following dynamics for  the  $N$-agent non exchangeable state process $\bX^N := (X^{i,N})_{i \in \llbracket 1, N \rrbracket}$: 
\begin{align}\label{eq : dynamic N agents}
\left\{\begin{aligned}  
       X_0^{i,N} &= x_0^i,   \\
       X_{t+1}^{i,N} &=  F_N( \frac{i}{N},X_t^{i,N},\alpha_t^{i,N}, \frac{1}{N} \sum_{j=1}^{N} \delta_{(\frac{j}{N},X_t^{j,N},\alpha_t^{j,N})},\epsilon_{t+1}^{i/N}, \epsilon_{t+1}^0), \quad t \in \N.
    \end{aligned}
        \right.
\end{align}
where $\bx_0 := (x_0^i)_{i \in \llbracket 1 , N \rrbracket} \in \Xc^N$ and the control followed by agent $i$  is adapted to the filtration $\F^N :=(\Fc_t^N)_{t \in \N}$ generated by $\boldsymbol{\epsilon}^N =\big((\epsilon_t^{i/N})_{i \in \llbracket 1, N \rrbracket}, \epsilon_t^0 \big)_{t \in \N^*}$ completed by a family of mutually i.i.d. uniform random variables $\boldsymbol{Z}^N = (Z_t^{i/N})_{i \in \llbracket 1, N \rrbracket, t \in \N}$ used for the randomization of the controls, i.e., $\Fc_t^N = \sigma \big( (\boldsymbol{\epsilon}_s^N)_{s \leq t} \big)  \vee \sigma(\boldsymbol{Z}^N)$.

We now follow mainly the results from \cite{motte2023quantitative}, and  state  the Bellman equation for the $N$-agent problem, viewed  as a MDP with state space $\Xc^N$, action space $A^N$, noise sequence $\boldsymbol{\epsilon}^N = (\boldsymbol{\epsilon}_t^N)_{t \in \N^*}$ with $\boldsymbol{\epsilon}_t^N := \big((\epsilon_t^{i/N})_{i \in \llbracket 1, N \rrbracket}, \epsilon_t^0 \big)$ valued in $E^N \times E^0$. We consider the set $\A^{OL}_{N}$ of randomized open-loop policies defined  as  sequences $\boldsymbol{\nu}^N = (\boldsymbol{\nu}^N_t)_{t \in \N}$ with $\bnu_0^N$ measurable function from $([0,1]^N)^{\N}$ into $A^N$ and $\bnu_t^N$ measurable function from $([0,1]^N)^{\N} \times (E^N \times E^0)^t$ into $A^N$ for $t \in \N^*$. Now, for each $\bnu^N \in \A^{OL}_{N}$, we associated a control $\alpha^{\bnu^N} \in \boldsymbol{\Ac}_N$ (recall its definition in \eqref{eq : def_admissible_policies_N_agents}) given  by 
\begin{align}
    \alpha_t^{\boldsymbol{\bnu}^N} := \boldsymbol{\nu}_t^N(\boldsymbol{Z}^N, (\boldsymbol{\epsilon}_s^N)_{s \in \llbracket 1, t \rrbracket}), \quad t \in \N.
\end{align}
The objective of the social planner is to maximize over the set $\boldsymbol{\Ac}_N$ of $A^N$-valued process $\balpha = (\alpha_t^{i,N})_{i \in \llbracket 1, N \rrbracket,t \in \N}$ the following criterion:
\begin{align}\label{eq : Expected Gain N-agent}
    V_N^{\balpha}(\bx_0) := \frac{1}{N} \sum_{i=1}^{N} \E \Big[ \sum_{t \in \N} \beta^t  f_N\big(\frac{i}{N} , X_t^{i,N}, \alpha_t^{i,N} , \frac{1}{N} \sum_{j=1}^{N} \delta_{(\frac{j}{N}, X_t^{j,N}, \alpha_t^{j,N})}\big) \Big],
\end{align}
and the value function of the $N-$agent non exchangeable MDP control problem as 
\begin{align}\label{eq : value function N agents}
    V_N(\bx_0) :=  \underset{\balpha \in \bAc_{N}}{\sup} V_N^{\balpha}(\bx_0).
\end{align}
We remark that the control problem \eqref{eq : dynamic N agents}-\eqref{eq : value function N agents} is a standard MDP with state space $\Xc^N$ and action space $A^N$ with randomized open-loop controls. As it is standard, well-posedness of the $N$-agent formulation is discussed briefly in Appendix \ref{appendix : bellman_equation_N_agents}. We refer to Appendix B in \cite{motte2023quantitative} for a deeper formulation of the setting.

\subsection{Additional assumptions for quantitative propagation of chaos}

We make the following assumptions on the state dynamics function $F_N$ and reward function $f_N$ of the $N$-agent problem. 
\begin{Assumption}\label{assumption: regularity_prop_chaos_Fn_fn_F_f}
   The functions 
   \begin{align}
        F_N : I \times \Xc \times A \times \Pc(I \times \Xc \times A) \times E \times E^0 \to \Xc, \qquad \quad
        f_N : I \times \Xc \times A \times \Pc(I \times \Xc \times A) \to \R,
   \end{align}
   are Borel measurable. Moreover, we have
\begin{itemize}
\item[(1)] 
\begin{align}
    \E \big[d \big(F_N(u,x,a,\mu,\epsilon^1_1, e^0),F_N(u,x',a',\mu' , \epsilon_1^1,e^0) \big) \big] \leq L_F \big ( \bd((x,a),(x,a')) + \Wc(\mu,\mu') \big).
\end{align}
 \item[(2)] 
\begin{align}
\big | f_N(u,x,a,\mu)  - f_N(u,x',a',\mu')  \big | \leq  L_f \big( \bd((x,a),(x',a')) + \Wc(\mu,\mu') \big) 
\end{align}
\end{itemize} 
for every  $u \in I$, $x,x'\in\Xc$, $a,a' \in A$, $\mu,\mu' \in \Pc(I \times \Xc \times A)$ and $e^0 \in E^0$ and where the constants $L_F$ and $L_f$ have been introduced in Assumption \ref{assumptions: F and f}. 
\end{Assumption}
\begin{Remark}
\normalfont
   The proposed setting slightly differs from those considered in \cite{motte2022mean, motte2023quantitative}, as we allow the state and reward functions of the $N$-agent system to depend explicitly on $N$. This choice is motivated by the recent literature on non-exchangeable mean field systems, where the interaction structure is typically described by a sequence of step-graphons $(G_N)_{N \in \N^{\star}}$ converging, in a suitable sense, to a limiting graphon $G$. For instance, one could look at the following sequence of $(f_N)_{N \in \N^{\star}}$ and $f$ in the form
   \begin{align}
        f_N(u,x,a,\mu) &= \int_{I \times \Xc \times A} G_N(u,x,a,v,x',a') \mu(\d x', \d a', \d v), \notag \\
        f(u,x,a,\mu) &= \int_{I \times \Xc \times A} G(u,x,a,v,x',a') \mu(\d x', \d a', \d v ).
\end{align}
\end{Remark}

The  dependence upon  $N$ in the state and reward functions requires us to have some explicit convergence assumption of $f_N$ towards $f$ and $F_N$ towards $F$ which are provided by the assumptions below.

\begin{Assumption}\textnormal{(Convergence of $f_N$ and $F_N$ towards $f$ and $F$).}\label{assumption : fN to f and FN to F}
There exists two decreasing positive sequences $(\epsilon_N^F)_{N \in \N^*}$ and $(\epsilon_N^f)_{N \in \N^*}$  satisfying $\underset{N \to \infty}{\text{ lim }} \epsilon_N^f = \underset{N \to \infty}{\text{ lim }} \epsilon_N^F = 0$ such that
    \begin{itemize}
        \item [(i)]
         \begin{align}
         \underset{(\bx,\ba) \in \Xc \times A}{\sup }\frac{1}{N} \sum_{j=1}^{N} | f( \frac{j}{N} , x^j,a^j,\mu_N \big[\bu,\bx,\ba \big]) - f_N(\frac{j}{N}, x^j,a^j,\mu_N \big[ \bu , \bx , \ba \big]) | \leq \epsilon_N^f
    \end{align}
       \item [(ii)]
        \begin{align}
         \underset{(\bx,\ba) \in \Xc \times A}{\sup }\E \Big[ \frac{1}{N} \sum_{i=1}^{N} d   \big(F(\frac{i}{N},  x^i,a^i , {\mu}_N^{\lambda} [\bu,\bx,\ba] , \epsilon_1^i, \epsilon_1^0),F_N(\frac{i}{N},  x^i,a^i , \mu_N^{\lambda} \big[\bu,\bx,\ba] , \epsilon_1^i, \epsilon_1^0) \big)  \Big]  \leq \epsilon_N^F .
    \end{align}
    \end{itemize}
    where we recall that $\mu_N[\bu,\bx,\ba]$, $\mu_N^{\lambda}[\bu,\bx,\ba]$ and $\bu$ have been defined in \eqref{eq : Empirical Measures def}.
\end{Assumption}
\begin{Remark}
\normalfont
    Notice that Assumptions \ref{assumption : fN to f and FN to F} is satisfied whenever
    \begin{align}
    \underset{ 1 \leq j \leq N}{\max} \underset{(x,a,\mu) \in \Xc \times A \times \Pc(I \times \Xc \times A)}{\sup}  | f( \frac{j}{N} , x,a,\mu) - f_N(\frac{j}{N}, x,a,\mu) \big | \leq \epsilon^f_N, \notag
\end{align}
and
\begin{align}
\underset{ 1 \leq j \leq N}{\max} \underset{(x,a,\mu) \in \Xc \times A \times \Pc(I \times \Xc \times A)}{\sup}
         \E \Big[ d   \big(F(\frac{j}{N},  x,a , \mu  , \epsilon_1^i, \epsilon_1^0),F_N(\frac{j}{N},  x,a , \mu , \epsilon_1^i, \epsilon_1^0) \big)  \Big]  \leq \epsilon^F_N \notag. 
\end{align}
\end{Remark}

\vspace{2mm}

\section{Main results}

\subsection{Bellman fixed point on $\Pc_{\lambda}(I \times \Xc)$}

Our first result is to show the equivalence between weak and strong formulations of the CNEMF-MDP, i.e.,  the equality between the value functions $V_{\text{S}}$ and $V_{\text{W}}$ relying on a lifting procedure on the space $\Pc_{\lambda}(I \times \Xc)$ inspired by similar arguments as used in \cite{motte2022mean}. 

\noindent  
\begin{Theorem}\textnormal{(Equivalence of value functions between weak and strong formulation).}\label{thm: equivalence_value_functions_weak_strong}

\noindent Let $\boldsymbol{\xi}=(\xi^u)_{u \in I}$ and $\xi$ be random variables satisfying the hypothesis of Proposition \ref{lemma: equivalence_between_weak_strong} implying  that  $\P_{\xi^u} = \P_{\xi |U=u}$ $\lambda-\text{a.e}$. Then, under the \textbf{randomization hypothesis} stated under the weak and strong formulation, we have
\begin{align}\label{eq : equality_strong_weak_controls}
    V_S(\boldsymbol{\xi}) = V_W(\xi) = \hat{V}(\mu),
\end{align}
where $\mu(\d u ,\d x) = \lambda(\d u )\P_{\xi^u}(\d x)  = \lambda(\d u )\P_{\xi | U=u}(\d x)$ and where $\hat{V}$ is  defined in \eqref{eq : value function canonical formulation}.
Moreover, there exists a map $\mathfrak{a}_0 \in L^0(\Pc_{\lambda}(I \times \Xc) \times I \times \Xc \times [0,1],A)$ called optimal randomized feedback policy such that the randomized feedback control defined  for $\lambda-\text{a.e } u \in I$ by  
\begin{align}
    \hat{\alpha}_t^{S,u} = \mathfrak{a}_0(\P^0_{X_t^v}(\d x) \d v , u,X_t,Z_t^u) , \quad \hat{\alpha}_t^{W} = \mathfrak{a}_0(\P^0_{(U,X_t)} ,U,X_t,Z_t), \quad \P-\text{a.s}, \quad t \in \N,
\end{align}
are  respectively optimal for $V_S(\boldsymbol{\xi})$ and $V_W(\xi)$.
\end{Theorem}

\begin{Remark}\label{rmk : bellman_fixed_point}
\normalfont
   Theorem \ref{thm: equivalence_value_functions_weak_strong} has several important implications for the underlying control problem, which can be summarized as follows.
    \begin{enumerate}
       \item[(1)] The CNEMF-MDP can be formulated as a standard MDP control problem with state space $\Pc_{\lambda}(I \times \Xc)$. Moreover, by \eqref{eq : equality_strong_weak_controls}, the control problem is law-invariant. In particular, let $\boldsymbol{\xi}=(\xi^u)_{u}, \boldsymbol{\xi'}=(\xi'^u)_u$, $\xi$ and $\xi'$ such that $\P_{\xi|U=u} = \P_{\xi^u} = \P_{\xi'^u}= \P_{\xi'|U=u}$  $\lambda$-a.e, then, one has $V_W(\xi) =V_S(\boldsymbol{\xi}) =V_S(\boldsymbol{\xi'}) = V_W(\xi') = \hat{V}(\mu)$ where $\mu(\d u ,\d x) = \lambda(\d u ) \P_{\xi^u}(\d x)$. Consequently, the value functions $V_S$ and $V_W$ can be identified with a single value function $\hat V$ defined on $\Pc_{\lambda}(I \times \Xc)$.

\item[(2)] The existence of an optimal randomized feedback control for the weak control problem is equivalent to the existence of an optimal randomized feedback control for the strong control problem. In both cases, the optimal control is given by the same feedback map $\mathfrak{a}_0$.

    \end{enumerate}
\end{Remark}

\subsection{Propagation of chaos for value functions}

Our second result is to quantify the rate of convergence of the value function of the $N$-agent MDP $V_N$ towards the value function of the CNEMF-MDP. Following Theorem \ref{thm: equivalence_value_functions_weak_strong} and Remark \ref{rmk : bellman_fixed_point}, we refer to $\hat V$ independently  for $V_{S}$ or $V_{W}$.
\begin{Theorem}\textnormal{(Quantitative propagation of chaos for the CNEMF-MDP).}\label{Thm : convergence value functions}
Under Assumptions \ref{assumptions: F and f}, \ref{assumption : measurability initial information}, \ref{assumption: regularity_prop_chaos_Fn_fn_F_f} and \ref{assumption : fN to f and FN to F},
    there exists some positive constant  $C$ such that for all  $\bx := (x^i)_{ i \in \llbracket 1, N \rrbracket} \in \Xc^N$, we have
    \begin{align}\label{eq: propagation_of_chaos}
     \Big | V_N(\bx ) -  \hat{V}( \mu_N^{\lambda}[\bu , \bx ]) \Big| \leq C \Big( M_N^{\gamma}  + \epsilon_N^f + (\epsilon_N^F)^{\gamma} 
        \Big),
\end{align}
where $\gamma = \text{min} \big(1,\frac{| \text{ln}(\beta)|}{ \text{ln}(2 L_f)} \big)$. 
\end{Theorem}

\begin{Remark}
\normalfont
    Theorem \ref{Thm : convergence value functions} provides a refined quantitative analysis of propagation of chaos compared to the classical mean field framework in \cite{motte2023quantitative}. In particular, the error bound in \eqref{eq: propagation_of_chaos} involves additional terms whose roles can be interpreted as follows.
\begin{enumerate}
    \item[(1)] In the CNEMF setting, the value function is defined on the space $\Pc_{\lambda}(I \times \Xc)$, which introduces an additional dimension corresponding to the agents’ labels. As a consequence, the quantity $M_N$, which depends on the dimension of the underlying state space, is directly affected.
    \item[(2)] The terms $\epsilon_N^f$ and $\epsilon_N^{F}$ are directly related to the convergence of $f_N$ and $F_N$ towards $f$ and $F$, respectively, in the $L^1$ sense. These error terms should be interpreted in light of interaction structures such as graphons: for instance, when assuming that a sequence of step graphons $(G_N)_{N \in \N^*}$ converges to a limiting graphon $G$ in $L^1$, these quantities naturally quantify the discrepancy induced by the finite-agent approximation.
\end{enumerate}

\end{Remark}

\subsection{Approximate optimal policies}

Our next results explain how to obtain approximate optimal policy for the $N$-agent MDP from $\epsilon$-optimal control of the CNEMF-MDP and how to quantify the error approximation. We shall rely on the notion of $\epsilon$-optimal randomized feedback policy which has been introduced in Definition \ref{def : lifted randomized feedback policy} of  Section \ref{subsec : optimal_randomized_feedback}. The proofs of these results follow as straightforward extensions of those in \cite{motte2023quantitative}, once we work on the label-state product space $I \times \Xc$ and have established the propagation of chaos result in Theorem \ref{Thm : convergence value functions}. We therefore refer the reader to the cited reference for full details.

\begin{Theorem}\label{thm : approximate optimal policy 1}
    Under the assumptions of Theorem \ref{Thm : convergence value functions}, let $\mathfrak{a}_{\epsilon}$ be an $\epsilon$-optimal randomized  feedback policy for the CNEMF-MDP.
    Then, there exists a measurable function $\bpi^{\mathfrak{a}_{\epsilon},N}$ from $\Xc^N$ to $A^N$, called feedback policy for the $N$-agent MDP, such that 
    \begin{align}\label{eq : construction_epsilon_optimal_feedback_policy_case_1}
        \boldsymbol{\pi}^{\ba_{\epsilon},N}(\bx) = \underset{\ba \in A^N}{\arg \min } \Wc \big( \P_{(\tilde{\xi}_{\bx},\boldsymbol{\mathfrak{a}}_{\epsilon}( \mu_N^{\lambda} [\bu, \bx],\tilde{\xi}_{\bx}, Z))}, \mu_N^{\lambda} [\bu, \bx, \ba ] \big), \quad \bx \in \Xc^N,
    \end{align}
    with $(\tilde{\xi}_{\bx},Z) \sim \mu_N^{\lambda}\big[\bu,\bx] \otimes \Uc([0,1])$ leading to a feedback control $\balpha^{\epsilon,N} =(\balpha_t^{\epsilon,N} := \bpi^{\mathfrak{a}_{\epsilon},N}(\bX_t)\big)_{t \in \N} \in \bAc_N$ which is 
     $\Oc( \epsilon + M_N^{\gamma} +  \epsilon_N^f + (\epsilon_N^F)^{\gamma})$, i.e
    \begin{align}
        V_N(\bx_0)  - C \Big( \epsilon + M_N^{\gamma}  + \epsilon_N^f  + (\epsilon_N^F)^{\gamma}  \Big) \; \leq \;   V_N^{\balpha^{\epsilon,N}}(\bx_0).
    \end{align}
\end{Theorem}
\begin{Theorem}\label{thm : approximate optimal policy 2}
    Under the assumptions of Theorem \ref{Thm : convergence value functions}, let $\mathfrak{a}_{\epsilon}$ be an $\epsilon$-optimal randomized feedback policy  for the CNEMF-MDP, assumed to satisfy the  regularity condition
    \small
    \begin{align}\label{eq : regularity_assumption_aeps}
        \E \Big[ d_A \big( \mathfrak{a}_{\epsilon}(\mu,u,x,Z), \mathfrak{a}_{\epsilon}(\mu,u',x',Z) \big) \Big] \leq  K \Big( | u-u'| + d( x,x') \Big), \quad \forall (u,x),(u',x') \in I \times \Xc, \hspace{0.3 cm} \mu \in \Pc_{\lambda}(I \times \Xc),
     \end{align}
     \normalsize
     where $Z \sim \Uc([0,1])$ and for some positive constant $K$. We consider the randomized feedback policy in the $N$-agent model by
    \begin{align}
        \bpi_{r}^{\mathfrak{a}_{\epsilon},N}(\bx,\bu) := \big(\mathfrak{a}_{\epsilon}(\mu_N^{\lambda}[\bu,\bx], \frac{i}{N},x^i,u^i )\big)_{i \in \llbracket 1, N \rrbracket}, \quad \bx :=(x^i)_{i \in \llbracket 1, N \rrbracket} \in \Xc^N, \bu = (u^i)_{i \in \llbracket 1, N \rrbracket} \in [0,1]^N.
    \end{align}
    Then, the randomized feedback control $\balpha^{r,\epsilon,N} = \big(\balpha_t^{r,\epsilon,N} := \bpi_r^{\mathfrak{a}_{\epsilon},N}(\bX_t, \boldsymbol{Z}_t) \big)_{t \in \N} \in \bAc_N$  
    where $\big \lbrace \boldsymbol{Z}_t =(Z_t^i)_{i \in \llbracket 1, N \rrbracket}, t \in \N \big \rbrace$ is a family of mutually  i.i.d uniform random variables on $[0,1]$, independant of $\boldsymbol{\epsilon}^N= \big( (\epsilon_t^{i/N})_{i \in \llbracket 1, N \rrbracket}, \epsilon_t^0 \big)_{t \in \N}$, is an $\Oc\big( \epsilon + M_N^{\gamma}   +  \epsilon_N^f + (\epsilon_N^F)^{\gamma} \big)$ optimal control for $V_N(\bx_0)$ i.e
    \begin{align}
        V_N(\bx_0)  - C (1+K) \big( \epsilon + M_N^{\gamma} + \epsilon_N^ f+ (\epsilon_N^F)^{\gamma}  \big) \; \leq  \; V_N^{\balpha^{r,\epsilon,N}}(\bx_0).
    \end{align}
\end{Theorem}

\begin{Remark}
\normalfont
\noindent 
    \begin{enumerate}
        \item [(1)] Theorem \ref{thm : approximate optimal policy 1} provides a way to compute approximate optimal policies for the $N$-agent MDP starting from $\epsilon$-optimal randomized feedback map $\mathfrak{a}_{\epsilon}$ for the CNEMF-MDP. Indeed, it says that given a state $\bx \in \Xc^N$, we are looking for the label-state-action triplet  $(\bu,\bx,\ba)$ such that their lifted empirical distribution $\mu_N^{\lambda}\big[\bu,\bx,\ba]$ is as close as possible to $\P_{(\tilde{\xi}_{\bx},\boldsymbol{\mathfrak{a}}_{\epsilon}( \mu_N^{\lambda} [\bu, \bx],\tilde{\xi}_{\bx}, Z))}$ in the Wasserstein distance. However, computing the argmin \eqref{eq : construction_epsilon_optimal_feedback_policy_case_1} at any time can be particularly difficult in practice.
        \item [(2)] In comparison with Theorem \ref{thm : approximate optimal policy 1}, Theorem \ref{thm : approximate optimal policy 2} proposes a natural and easy procedure to compute approximate optimal policies for the $N$-agent MDP if we are given an $\epsilon$-randomized feedback policy $\mathfrak{a}_{\epsilon}$ satisfying the regularity assumption \eqref{eq : regularity_assumption_aeps} for the CNEMF-MDP. Indeed, it says that it is enough to input into the measure argument of the $\epsilon$-optimal randomized feedback map  $\mathfrak{a}_{\epsilon}$ the empirical lifted label-state distribution of the $N$-agent MDP instead of its theoretical label-state distribution from the CNEMF-MDP. Moreover, instead of inputting into the label-state argument of $\mathfrak{a}_{\epsilon}$ the theoretical label-state $(U,X_t)$, we input the $N$-agent individual states with label $\frac{i}{N}$ for every agent $i \in \llbracket 1, N \rrbracket$, and moreover, we use a randomization  at any time and for any agent.
    \end{enumerate}
\end{Remark}

\section{Proof of main results}

This section is devoted to the proof of the main results stated above. As the CNEMF-MDP can be viewed as a standard McKlean-Vlasov control problem on the label-state product space $I \times \Xc$ with marginal constraint on $I$, we will mostly rely on the existing results borrowed from \cite{motte2022mean} and \cite{motte2023quantitative}.

\subsection{Proof of Theorem \ref{thm: equivalence_value_functions_weak_strong}}

\subsubsection{Lifted MDP on $\Pc_{\lambda}(I \times \Xc)$ in the weak formulation}

We now show how we can lift the CNEMF-MDP  to a classical MDP on the space of probabilities $\Pc_{\lambda}(I \times \Xc)$.  We set  $\F^0 = (\Fc_t^0)_{t \in \N^*}$ the filtration generated by the common noise $\Fc_t^0 = \sigma \big( (\epsilon_s^0)_{s \leq t}\big)$. Let $\alpha \in  \Ac^{W}$ be an open-loop control, $\xi$  an initial state $\sigma(U) \vee \Gc$ measurable,  and its associated state process $X=\big(X^{\alpha,\xi}_t)_{t \in \N} $ given by the dynamics \eqref{eq:dynamics_NEMFC_weak_formulation} . We  denote by $\lbrace \mu_t = \P^0_{(U,X_t)}, t \in \N \rbrace$ the random $\Pc_{\lambda}(I \times \Xc)$-valued process, which is $\F^0$-adapted by Proposition A.1 in \cite{motte2022mean}.

\begin{Lemma}

\noindent Define the measurable map $\tilde{F}$ on $I \times \Xc \times A \times \Pc(I \times \Xc \times A) \times E \times E^0 \to I \times \Xc$ as 
\begin{align}\label{eq : definition tilde F}
    \tilde{F}(u,x,a,\mu,e,e^0) := \big(u , F(u,x,a,\mu,e,e^0)\big).
\end{align}
Then, 
\begin{align}\label{eq : dynamics of mu}
    \mu_{t+1} = \tilde{F}(\cdot, \cdot, \cdot, \P^0_{(U,X_t,\alpha_t)},  \cdot,\epsilon_{t+1}^0) \sharp \big(\P^0_{(U,X_t,\alpha_t)} \otimes \lambda_{\epsilon} \big) \quad \text{$\P$-a.s},
\end{align}
\end{Lemma}

\begin{proof}
   The results follows  by definition of $X_{t+1}$ in \eqref{eq:dynamics_NEMFC_weak_formulation}, the fact that $\big(\P^0_{(U,X_t,\alpha_t)}\big)_{t \in \N}$ is an $\F^0$-adapted process (see Proposition A.1 in \cite{motte2022mean}) and after noticing that $\epsilon_{t+1}$ is independent of the random vector $(U,X_t,\alpha_t, \epsilon^0_{t+1})$, which implies $\P^0_{(U,X_t,\alpha_t, \epsilon_{t+1})} = \P^0_{(U,X_t,\alpha_t)} \otimes \lambda_{\epsilon}$.
\end{proof}
Similarly to the approach in \cite{motte2022mean}, we now consider the $\F^0$-adapted  control  process to be $\balpha_t = \P^0_{(U,X_t,\alpha_t)}$ valued in the space of probability measures $\boldsymbol{A} := \Pc(I \times \Xc \times A)$, endowed with its natural $\sigma$-algebra with the Wasserstein metric. We note that this $\boldsymbol{A}$-valued process has to satisfy the property that $\pr_{12 } \sharp \balpha_t = \mu_t$ at any $t \in \N$.

The marginal constraint is handled by the coupling argument shown in Lemma 4.1 in \cite{motte2022mean}. We shall denote by $\zeta \in L^0(\Pc_{\lambda}(I \times \Xc)^2 \times I \times \Xc \times [0,1] ; I \times \Xc)$ the measurable coupling which enables the construction of a measurable map $\bp : \Pc_{\lambda}(I \times \Xc) \times \bA \to \bA$ as 
\begin{align}\label{eq : coupling_projection}
    \bp(\mu,\ba) := \P_{( \zeta(\pr_{12} \sharp\ba, \mu,U,\xi,Z),\alpha_0 )}, \quad \mu \in \Pc_{\lambda}(I \times \Xc), \hspace{0.2 cm} \ba \in \bA ,
\end{align}
where $((U,\xi),\alpha_0) \sim \ba$, and $Z$ is a uniform random variable independent of $(U,\xi,\alpha_0)$. We recall that $\bp$ is such that for all $(\mu,\boldsymbol{a}) \in \Pc_{\lambda}( I \times \Xc) \times \boldsymbol{A}$, we have $\pr_{12} \sharp\bp(\mu,\ba) = \mu, \quad \text{ and if } \pr_{12}\sharp \ba = \mu, \text{ then } \bp(\mu,\ba) = \ba$.
By using the coupling projection $\bp$, the dynamics \eqref{eq : dynamics of mu} can be rewritten as 
\begin{align}\label{eq : dynamics_lifted_weak_formulation}
    \mu_{t+1} = \hat{F} \big(\mu_t, \balpha_t,\epsilon_{t+1}^0 \big) \quad \P-\text{a .s}, \quad t \in \N,
\end{align}
where the function $\hat{F} : \Pc_{\lambda}(I \times \Xc) \times \bA \times E^0 \to \Pc_{\lambda}(I \times \Xc)$ is defined by
\begin{align}\label{eq : definition hat F}
    \hat{F}(\mu,\ba,e^0) := \tilde{F}(\cdot,\cdot,\cdot, \boldsymbol{p}(\mu,\ba), \cdot,e^0)\sharp \big( \boldsymbol{p}(\mu,\ba) \otimes \lambda_{\epsilon} \big),
\end{align}
where $\tilde{F}$ has been introduced in \eqref{eq : definition tilde F}. $\hat{F}$ is clearly measurable as $\tilde{F}$ and $\bp$ are measurable.
Let's define the measurable function $\hat{f} : \Pc_{\lambda}(I \times \Xc) \times \boldsymbol{A}  \to \R$ by 
\begin{align}\label{eq : hatf_weak_lifted_weak_formulation}
    \hat{f}(\mu,\ba) := \int_{I \times \Xc \times A} f \big(u,x,a, \boldsymbol{p}(\mu,\ba) \big) \boldsymbol{p}(\mu,\ba)(\d u , \d x ,\d a).
\end{align}
Now, the MDP with characteristics $\big(\Pc_{\lambda}(I \times \Xc), \bA = \Pc(I \times \Xc \times A), \hat{F},\hat{f}, \beta \big)$ is therefore well posed. We now define $\bAc$ the set of $\F^0-$adapted processes valued in $\bA$ and given an open-loop control $\bnu =(\bnu_t)_{t \in \N} \in \bAc$, we consider the controlled dynamics
\begin{align}\label{eq : dynamics canonical formulation}
\begin{cases}
    \mu_{t+1} &= \hat{F}(\mu_t,\bnu_t,\epsilon_{t+1}^0), \quad \P {\text{ a.s}},\quad t \in \N, \\
     \mu_0 &= \mu \in \Pc_{\lambda}(I \times \Xc),
\end{cases}
\end{align}
with associated value function
\begin{align}\label{eq : value function canonical formulation}
    \hat{V}^{\bnu}(\mu) := \E \Big[ \sum_{t \in \N}  \beta^t \hat{f}(\mu_t,\bnu_t) \Big], \quad \hat{V}(\mu) := \underset{\bnu \in \bAc}{\sup} \hat{V}^{\bnu}(\mu).
\end{align}
Given an initial condition $\xi$, we also define the relaxed operator $\Lc_{\xi}^0$ associated to the control space $\bA$ as follows 
\begin{align}\label{eq : operators Lxi0}
\begin{cases}
  \mathcal{L}_{\xi}^0 : \mathcal{A}^{W} &\longrightarrow \bAc,  \\ \quad
\alpha = (\alpha_t)_{t} &\longmapsto \balpha^{R} = (\balpha_t^R)_t 
: \balpha_t^R:= \P^0_{(U,X_t,\alpha_t)}, \quad 
t \in  \N ,  \\
\end{cases}
\end{align}
In Appendix \ref{subsec: lifted_strong_formulation}, we show similarly how we can lift the strong control problem on $\Pc_{\lambda}(I \times \Xc)$ following the same steps. Given $\xi$, $\boldsymbol{\xi}$, $\alpha \in \Ac^{W}$ and $\balpha \in \Ac^S$ satisfying the Assumptions of Proposition \ref{lemma: equivalence_between_weak_strong}, then it follows from Proposition \ref{lemma: equivalence_between_weak_strong} that $\balpha^{R} := \Lc_{\xi}^0(\alpha) =\Lc^0_{\boldsymbol{\xi}}(\balpha) $, where the map $\Lc_{\boldsymbol{\xi}}^0$ has been introduced in \eqref{eq : Lxi mapping_strong_formulation} and the process $\mu_t = \P^0_{(U,X_t)}=\P^0_{X_t^v}(\d x) \d v $, $t \in \N,$ follows the dynamics \eqref{eq : dynamics_lifted_weak_formulation}. Moreover, using the definition of $\hat{f}$ in \eqref{eq : hatf_weak_lifted_weak_formulation}, we show easily by taking the conditional expectation with respect to $(\epsilon^0_t)_{t \in \N^{\star}}$ in \eqref{eq : cost functional V weak formulation} and following \eqref{eq : equivalent_formulation_expected_gain} that 
\begin{align}\label{eq : contrlweakhatV}
    V_W^{\alpha}(\xi) = V^{\balpha}_S(\boldsymbol{\xi})= \hat{V}^{\balpha^{R}}(\mu), \quad \text{where } \mu(\d u , \d x ) = \P_{(U,\xi)}= \P_{\xi^u}(\d x) \d u \in \Pc_{\lambda}(I \times \Xc),
\end{align}
This clearly shows that $V_W(\xi) \leq \hat{V}(\mu)$ and $V_S(\boldsymbol{\xi}) \leq \hat{V}(\mu)$ for $\mu = \P_{{(U,\xi)}}=\P_{\xi^u}(\d x) \d u  \in \Pc_{\lambda}(I \times \Xc)$ as a consequence of the space enlargement of the control space.

The goal is now to prove the reverse inequality  and to characterize at the same time the value functions $V_W$ and $V_S$, i.e the solution to the CNEMF-MDP problem through a Bellman fixed point equation.

\subsubsection{Bellman fixed point on $\Pc_{\lambda}(I \times \Xc)$}

We now study the Bellman equation corresponding to the lifted MDP \eqref{eq : dynamics canonical formulation}-\eqref{eq : value function canonical formulation} on $\Pc_{\lambda}(I \times \Xc)$ by working on the canonical space $(E^0)^{\N}$. We denote the canonical identity function on $(E^0)^{\N}$ as $\epsilon^0$ and by $\epsilon_t^0$ its $t-$th projection in $(E^0)^{\N}$, ie the map defined as $\epsilon_t^0(\omega) =\omega_t$ where $\omega= (\omega_0,\ldots, \omega_{t-1},\omega_t, \omega_{t+1}, \ldots) \in (E^0)^{\N}$. The canonical space $(E^0)^{\N}$ is endowed by the $\sigma$-algebra generated by the canonical projection maps ie $\sigma\big(\epsilon_1^0,\ldots, \epsilon_t^0, \ldots \big)$ and endowed with the measure on $(E^0)^{\N}$ constructed from the law of $\epsilon_1^0$. An open-loop control $\bnu \in \bAc$ as a sequence of $(\bnu_t)_{t \in \N}$ where $\bnu_t$ is a measurable map from $(E^0)^t$ to $\bA$ with convention that $\bnu_0$ is a constant.  Given $\boldsymbol{\nu} \in \bAc$ and $e^0 \in E^0$, we define $\vec{\bnu}^{e^0} := (\vec{\bnu}^{e^0}_t)_{t \in \N} \in \bAc$, where $\vec{\bnu}_t^{e^0}(\cdot) := \bnu_{t+1}(e^0,\cdot), t \in \N$. We also introduce the shifted operator $\theta$ on $(E^0)^{\N}$ as the map $\theta((\epsilon^0_t)_{t \in \N}) = (\epsilon_{t+1}^0)_{t \in \N}$. Given $\mu \in \Pc_{\lambda}(I \times \Xc)$ and $\bnu \in \bAc$, we denote by $(\mu_t^{\mu,\bnu})_{t \in \N}$ the solution to \eqref{eq : dynamics canonical formulation} and following the  same arguments in \cite{motte2022mean} (section 4.2), one can show the following  DPP relation
\begin{align}\label{eq : DPP principle}
    \hat{V}^{\bnu}(\mu) =  \hat{f}(\mu,\bnu_0) + \beta \E \big[ \hat{V}^{\vec{\bnu}^{\epsilon_1^0}}(\mu_1^{\mu,\bnu}) \big]
    \end{align}
We then introduce the Bellman operator associated to this non exchangeable control problem $\Tc : L_m^{\infty}(\Pc_{\lambda}(I \times \Xc)) \to L^{\infty}(\Pc_{\lambda}(I \times \Xc))$ defined for any $W \in L^{\infty}_m(\Pc_{\lambda}(I \times \Xc))$ by 
\begin{align}\label{eq : Bellman operator def 1}
    [\Tc W] (\mu) := \underset{\ba \in \bA}{\sup}  \big \lbrace \hat{f}(\mu,\ba) + \beta  \E \big[ W(\hat{F}(\mu, \ba,\epsilon_1^0)] \big] \big \rbrace, \quad \mu \in \Pc_{\lambda}(I \times \Xc).
\end{align}
As mentioned in \cite{motte2022mean}, the supremum in \eqref{eq : Bellman operator def 1} may lead to a non measurable function so it may not be a well defined operator. However,  this issue is handled by showing that $\Tc$ preserves the following space 
\small
\begin{align}
    \Mc := \big \lbrace  W \in L^{\infty}_m(\Pc_{\lambda}(I \times \Xc)) : \big| W(\mu) - W(\mu') \big| \leq 2 L_f \sum_{t =0}^{\infty} \beta^t  \textnormal{min} \big[(2 L_F)^t \Wc(\mu,\mu'), \Delta_{I \times \Xc}  \big], \hspace{0.1 cm} \forall (\mu,\mu') \in \Pc_{\lambda}(I \times \Xc) \big \rbrace.
\end{align}
\normalsize
which defines a complete metric as a closed subspace of the complete metric space $\Big(L^{\infty}_m\big(\Pc_{\lambda}(I \times \Xc)\big); \lVert \cdot \rVert_{\infty} \Big)$. Moreover, there exists a constant $K_{\star}$ (depending on $L_F$, $L_f$, $\beta$ and $\Delta_{I \times \Xc}$) such that  any function $W \in \Mc$ is $\gamma$-Hölder with constant factor $K_{\star}$, i.e
\begin{align}\label{eq : gamma_Holder_property}
    \big | W(\mu) - W(\mu') \big | \leq  K_{\star} \Wc(\mu,\mu')^{\gamma},
\end{align}
where we recall that $\gamma$ has been introduced in Theorem \ref{Thm : convergence value functions}.  Moreover, the Bellman operator $\Tc$ is a $\beta$-contraction on $L^{\infty}_m \big(\Pc_{\lambda}(I \times \Xc) \big)$ and admits a unique fixed point $V^{\star} \in \Mc$, which is  therefore $\gamma$-Hölder.

We can now state a first relation between the value function $\hat{V}$ and the fixed point of the Bellman operator $\Tc$ $V^*$ as follows
\begin{Lemma}\label{lemma : inequality_value_func_fixed_point}
    For any $\mu \in \Pc_{\lambda}(I \times \Xc)$, we have $\hat{V}(\mu) \leq V^*(\mu)$.
\end{Lemma}
\begin{proof}
    Let $\bnu \in \bAc$. Then, from the dynamic programming equation \eqref{eq : DPP principle} and since $\Tc V^* = V^*$, we have
    \begin{align}
        \underset{\mu \in \Pc_{\lambda}(I \times \Xc)}{\inf} \big (V^*(\mu) - \hat{V}^{\bnu}(\mu) \big)
        \geq &\underset{\mu \in \Pc_{\lambda}(I \times \Xc)}{\inf} \Big( \Tc V^*(\mu)  - \big( \tilde{f}(\mu,\bnu_0) + \beta \E \big[V^*(\mu_1^{\mu,\bnu}) \big] \big) + \beta \E \big[V^*(\mu_1^{\mu,\bnu} ) - \hat{V}^{\vec{\nu}^{\epsilon_1^0}}(\mu_1^{\mu,\bnu}) \big] \Big) \notag \\
        \geq &\beta  \underset{\mu \in \Pc_{\lambda}(I \times \Xc)}{\inf} (V^*(\mu) - \hat{V}^{\bnu}(\mu) \big)
    \end{align}
    where in the second inequality we used the definition of $\Tc$ 
    which shows that $\hat{V}^{\bnu}(\mu) \leq V^*(\mu), \quad \forall \mu \in \Pc_{\lambda}(I \times \Xc)$.
    Taking the supremum over $\bnu \in \bAc$ gives the required result.
\end{proof}
We now give some other representations of the operator $\Tc$. To this extend, we introduce the map $\bar{F} : \Pc_{\lambda}(I \times \Xc) \times L^0\big(I \times \Xc; \Pc(A) \big) \times E^0 \to \Pc_{\lambda}(I \times \Xc)$ as
\begin{align}
    \bar{F}(\mu, \hat{a}, e^0) = \tilde{F}(\cdot,\cdot,\cdot, \mu  \hat{a}, \cdot, e^0) \sharp \big( \mu  \hat{a} \otimes \lambda_{\epsilon} \big), \quad (\mu,\hat{a},e^0) \in \Pc_{\lambda}(I \times \Xc) \times L^0(I \times \Xc ; \Pc(A) \big) \times E^0,
\end{align}
where we recall that $\tilde{F}$ has been introduced in \eqref{eq : definition tilde F}. We also introduced the map $\bar{f} :  \Pc_{\lambda}(I \times \Xc) \times L^0(I \times \Xc ; \Pc(A)) \to \R $ as
\begin{align}
\bar{f}(\mu,\hat{a}) = \int_{I \times \Xc \times A} f(u,x,a, \mu  \hat{a}) (\mu  \hat{a})(\d u , \d x , \d a ).
\end{align}

\begin{Proposition}\label{prop : Bellman equation}
    For any $W \in L^{\infty}_m \big(\Pc_{\lambda}(I \times \Xc)\big)$ and $\mu \in \Pc_{\lambda}(I \times \Xc)$, we have
    \begin{align}\label{eq : bellman_equations}
        \big[ \Tc W \big](\mu) = \underset{\hat{a} \in L^0(I \times \Xc ; \Pc(A))}{\sup} \big[ \bar{\Tc}^{\hat{a}}W \big](\mu)= \underset{\mathrm{a} \in L^0( I \times \Xc \times [0,1];A)}{\sup} \big[ \T^{\mathrm{a}} W \big](\mu),
    \end{align}
where $\bar{\Tc}^{\hat{a}}$ and $\T^{\mathrm{a}}$  are operators defined on $L^{\infty}\big(\Pc_{\lambda}(I \times \Xc)\big)$ by 
\begin{align}\label{eq : definition Bellman operator before sup}
\begin{cases}
    \big[ \bar{\Tc}^{\hat{a}}W](\mu) &:= \bar{f}(\mu,\hat{a}) + \beta  \E \big[ W(\bar{F}(\mu,\hat{a},\epsilon_1^0)) \big]  \\
    \big[ \T^{\mathrm{a}} W \big](\mu) &:=  \E \Big[f((U,\xi), \mathrm{a}((U,\xi),Z), \P_{\big( (U,\xi),\mathrm{a}((U,\xi),Z) \big)} + \beta  W \big( \P^0_{\tilde{F}((U,\xi),\mathrm{a}((U,\xi),Z),\P_{(\xi,\mathrm{a}((U,\xi),Z)},\epsilon_1,\epsilon_1^0)}\big) \Big]  
\end{cases}
\end{align}
for any $\big((U,\xi),Z\big) \sim \mu \otimes \Uc([0,1])$. Moreover, we have
\begin{align}\label{eq : bellman operator form 4}
    \big[ \Tc W \big](\mu) = \underset{\alpha_0 \in L^0( \Omega;A)}{\sup} \E \Big[ f \big( (U,\xi),\alpha_0 , \P_{((U,\xi),\alpha_0)}\big) + \beta W \big(\P^0_{\tilde{F}((U,\xi),\alpha_0, \P_{((U,\xi),\alpha_0)},\epsilon_1,\epsilon_1^0)}\big) \Big].
\end{align}
\end{Proposition}

\begin{proof}
The proof of this  result is a simple extension of the one presented in Proposition 4.1 in \cite{motte2022mean} so we omit it.
\end{proof}
\begin{Remark}
\normalfont
    The proof of Proposition \ref{prop : Bellman equation} essentially follows from the marginal constraint of $\bp(\mu,\ba)$ which allows the construction of a kernel map $\hat{a} \in L^0\big(I \times \Xc ; \Pc(A) \big)$ satisfying $\bp(\mu,\ba) = \mu  \hat{a}$. Then, the construction of a suitable random variable follows from the randomization lemma in \cite{kallenberg2002foundations}.
\end{Remark}

\subsubsection{The subclass of optimal randomized feedback controls}\label{subsec : optimal_randomized_feedback}
In order to show that $\hat{V} = V^*$, i.e the value function $\hat{V}$ of the general lifted MDP on $\Pc_{\lambda}(I \times \Xc)$ satisfies the Bellman equation $\hat{V}=\Tc \hat{V}$ and the existence of an optimal control for $\hat{V}$, we will consider a suitable subclass of controls. To this extend, for any $\bpi \in L^0(\Pc_{\lambda}(I \times \Xc);\bA)$, called feedback policy, we introduce the associated Bellman operator $\Tc^{\bpi}$ on $L^{\infty}(\Pc_{\lambda}(I \times \Xc))$, defined for any $W \in L^{\infty}(\Pc_{\lambda}(I \times \Xc))$ by
\begin{align}\label{eq : bellmanoperatorfeedbackpolicies}
    \big[ \Tc^{\bpi} W](\mu) := \hat{f}(\mu,\bpi(\mu)) + \beta \E \big[ W( \hat{F} \big(\mu, \bpi(\mu), \epsilon_1^0) \big) \big], \quad \mu \in \Pc_{\lambda}(I \times \Xc).
\end{align}
We recall that  for any $\bpi \in L^0( \Pc_{\lambda}(I \times \Xc);\bA)$, the operator $\Tc^{\bpi}$ is a $\beta$-contraction on $L^{\infty}(\Pc_{\lambda}(I \times \Xc))$  and admits a unique fixed point $\hat{V}^{\bpi}$. Moreover, from standard theory on MDP, the unique fixed point $\hat{V}^{\bpi}$ associated to the feedback operator $\Tc^{\bpi}$ is equal to  
\begin{align}\label{eq : fixed_point_feedback_operator}
    \hat{V}^{\bpi}(\mu) = \E \Big[ \sum_{t \in \N} \beta^t \hat{f}(\mu_t, \bpi(\mu_t) ) \Big],
\end{align}
where $(\mu_t)_{t \in \N}$ is the MDP from \eqref{eq : dynamics canonical formulation}  with the feedback control control $\bnu^{\bpi} = (\bnu^{\bpi}_t)_{t \in \N} \in \bAc$ defined for any $t \in \N$ by $\bnu_t^{\bpi} = \bpi(\mu_t)$. In the following, we identify $\hat{V}^{\mathfrak{\bpi}}$ with $\hat{V}^{\bnu^{\bpi}}$ as defined in \eqref{eq : value function canonical formulation}.  

In order to solve the CNEMF-MDP, we introduce the subclass of lifted feedback randomized policies.

\begin{Definition}\label{def : lifted randomized feedback policy}
    A feedback policy $\bpi  \in L^0(\Pc_{\lambda}(I \times \Xc);\bA)$ is a lifted randomized feedback policy if there exists a measurable function $\mathfrak{a} \in L^0( \Pc_{\lambda}(I \times \Xc) \times I \times \Xc \times [0,1],A)$ called randomized feedback policy such that $(U,\xi, \mathfrak{a}(\mu,(U,\xi),Z)) \sim \bpi(\mu)$, for all $\mu \in \Pc_{\lambda}(I \times \Xc)$, with $((U,\xi),Z) \sim \mu \otimes \Uc([0,1])$.   
\end{Definition}
\begin{Remark}\label{rmk ; optimal_control_representation_weak_strong}
\normalfont
    Given $\mathfrak{a} \in L^0(\Pc_{\lambda}(I \times \Xc) \times I \times \Xc \times [0,1];A)$,  we denote by $\bpi^{\mathfrak{a}} \in L^0(\Pc_{\lambda}(I \times \Xc);\bA)$ the associated lifted randomized feedback policy, ie $\bpi^{\mathfrak{a}}(\mu) = \P_{(U,\xi, \mathfrak{a}(\mu,(U,\xi),Z))}$, for $\mu \in \Pc_{\lambda}(I \times \Xc)$ and $((U,\xi) , Z) \sim \mu \otimes \Uc([0,1])$. By definition of the Bellman operator $\Tc^{\bpi^{\mathfrak{a}}}$ in \eqref{eq : bellmanoperatorfeedbackpolicies} and since $\text{pr}_{12} \sharp \bpi^{\mathfrak{a}}(\mu) = \mu$, we have $p(\mu,\bpi^{\mathfrak{a}}(\mu)) = \bpi^{\mathfrak{a}}(\mu)$ and recalling \eqref{eq : definition Bellman operator before sup}, we see after defining the map $\mathfrak{a}^{\mu} = \mathfrak{a}(\mu, \cdot,\cdot , \cdot) \in L^0(I \times \Xc \times [0,1];A)$, that for any $W \in L^{\infty}(\Pc_{\lambda}(I \times \Xc))$
    \begin{align}\label{eq : bellman_operator_equalities_feedback_policies}
        \big[ \Tc^{\bpi^{\mathfrak{a}}} W](\mu) =  \big[ \T^{\mathfrak{a}^{\mu}} W \big](\mu), \quad \mu \in \Pc_{\lambda}(I \times \Xc).
    \end{align}
    Let $\boldsymbol{\xi} :=(\xi^u)_{u \in I}$ and $\xi$ satisfying the Assumptions of Proposition \ref{lemma: equivalence_between_weak_strong}. We then define $\alpha_W^{\mathfrak{a}} =(\alpha_t^{\mathfrak{a}})_{t \in \N} \in \Ac^{W}$ and $\balpha_S^{\mathfrak{a}} = (\alpha_t^u)^{\mathfrak{a}}_{t \in \N, u \in I} \in \Ac^S$ as for $\lambda-\text{a.e}$ $u \in I$
    \begin{align}\label{eq : randomized_feedback_controls_strong_weak}
    \begin{cases}
        (\alpha_t^u)^{\mathfrak{a}} &= \mathfrak{a}(\P^0_{X_t^v}(\d x) \d v , u, X_t^u, Z_t^u), \\
        \alpha_t^{\mathfrak{a}} &= \mathfrak{a}(\P^0_{(U,X_t)}, U,X_t, Z_t).
    \end{cases}
    \end{align}
    By construction, recalling the definitions of the operators $\Lc_{\xi}^0$ and $\Lc_{\boldsymbol{\xi}}^0$ respectively in \eqref{eq : operators Lxi0} and \eqref{eq : Lxi mapping_strong_formulation} and \eqref{eq : contrlweakhatV}, we can define  the associated lifted control $\balpha^{R,\mathfrak{a}} := \Lc_{\xi}^0( \alpha^{\mathfrak{a}}_{W}) = \Lc_{\boldsymbol{\xi}}^0(\balpha^{\mathfrak{a}}_S) $ which satisfies by construction $\balpha_t^{R,\mathfrak{a}} = \P^0_{(U,X_t, \alpha_t^{\mathfrak{a}})} = \bpi^{\mathfrak{a}}(\P^0_{(U,X_t)})$ for any  $t \in \N$. Following \eqref{eq : fixed_point_feedback_operator}, Equation \eqref{eq : value function canonical formulation} and Proposition \ref{lemma: equivalence_between_weak_strong}, we have $V_W^{\alpha^{\mathfrak{a}}_W}(\xi) = V^{\balpha^{\mathfrak{a}}_S}_{S}(\boldsymbol{\xi}) = \hat{V}^{\bnu^{\bpi^{\mathfrak{a}}}}(\mu)$, where $\mu = \P_{(U,\xi)} = \P_{\xi^u}(\d x) \d u $.
\end{Remark}

We now show a verification theorem type result for the general lifted MDP on $\Pc_{\lambda}(I \times \Xc)$ and hence for the CNEMF-MDP.
\begin{Proposition}\label{proposition : verification theorem}\textnormal{(Verification result).}
    Let $\epsilon \geq 0$. Suppose that there exists an $\epsilon$-optimal feedback policy $\bpi_{\epsilon} \in L^0(\Pc_{\lambda}(I \times \Xc);\bA)$ for $V^*$, i.e for any $\mu \in \Pc_{\lambda}(I \times \Xc)$
    \begin{align}\label{eq : verification result feedback policy}
       V^*(\mu) \leq \Tc^{\bpi_{\epsilon}}V^*(\mu) + \epsilon  
    \end{align}
    Then,  $ \big(\bnu_t^{\mathfrak{\pi}_{\epsilon}} =\bpi_{\epsilon}(\mu_t) \big)_{t \in \N} \in \bAc$ is $\frac{\epsilon}{1-\beta}$ optimal for $\hat{V}$ ie $\hat{V}^{\bpi_{\epsilon}} \geq \hat{V} - \frac{\epsilon}{1 - \beta}$ and we have $\hat{V} \geq V^* - \frac{\epsilon}{1- \beta}$. Furthermore, if $\bpi_{\epsilon}$ is a lifted randomized feedback policy according to Definition \ref{def : lifted randomized feedback policy} for a measurable function $\mathfrak{a}^{\epsilon} \in L^0(\Pc_{\lambda}(I \times \Xc) \times I \times \Xc \times [0,1];A)$, then recalling the notations introduced in \eqref{eq : randomized_feedback_controls_strong_weak} 
    \begin{enumerate}
        \item [(1)] $\balpha^{\mathfrak{a}_{\epsilon}}_S \in \Ac^S$ and $\alpha^{\mathfrak{a}_{\epsilon}}_{W} \in \Ac^W$ are respectively $\frac{\epsilon}{1- \beta}$ optimal for $V_S(\boldsymbol{\xi})$ and $V_W(\xi)$, i.e, we have $V_S^{\balpha^{\mathfrak{a}_{\epsilon}}_S}(\boldsymbol{\xi}) \geq V_S(\boldsymbol{\xi}) - \frac{\epsilon}{1- \beta}$ and $V_W^{\alpha^{\mathfrak{a}_{\epsilon}}_{W}}(\xi) \geq V_W(\xi) - \frac{\epsilon}{1- \beta}$
        \item [(2)]  Moreover, we have $V_S(\boldsymbol{\xi}) \geq V^{\star}(\mu) - \frac{ 2 \epsilon}{1-\beta}$ and $V_W(\xi) \geq V^{\star}(\mu) - \frac{2 \epsilon}{1- \beta}$, where $\mu = \P_{(U,\xi)} = \P_{\xi^u}(\d x) \d u$.
    \end{enumerate}
\end{Proposition}
\begin{proof}
   The first part of the proof follows similar argument as in Proposition 4.3 in \cite{motte2022mean}. Indeed, we use $\Tc^{\pi_{\epsilon}}\hat{V}^{\mathfrak{\pi}_{\epsilon}} = \hat{V}^{\pi_{\epsilon}}$, where $\hat{V}^{\pi_{\epsilon}}$ is defined in \eqref{eq : fixed_point_feedback_operator}, that $\Tc^{\mathfrak{\pi}_{\epsilon}}$ is a $\beta$-contraction on $L^{\infty}(\Pc_{\lambda}(I\times \Xc))$ and that $V^{\star} \geq \hat{V} \geq \hat{V}^{\mathfrak{\pi}_{\epsilon}}$ where the first inequality follows from Lemma \ref{lemma : inequality_value_func_fixed_point}. 
   If we assume furthermore that $\bpi_{\epsilon} := \bpi^{\mathfrak{a}_{\epsilon}}$ in the sense introduced in Remark \ref{rmk ; optimal_control_representation_weak_strong}, then by Remark \ref{rmk ; optimal_control_representation_weak_strong}, we have $V_W^{\alpha^{\mathfrak{a}_{\epsilon}}_W}(\xi) = V^{\balpha^{\mathfrak{a}_{\epsilon}}_S}_{S}(\boldsymbol{\xi}) = \hat{V}^{\bnu^{\bpi^{\mathfrak{a}_{\epsilon}}}}(\mu)$, where $\mu = \P_{(U,\xi)} = \P_{\xi^u}(\d x) \d u $. 
   Now, $(1)$ follow recalling the inequalities  $V_S(\boldsymbol{\xi})\leq \hat{V}(\mu)$ and $V_W(\xi) \leq \hat{V}(\mu)$ combined with the first inequality of the first assertion. 
   $(2)$ follows from $V_S(\boldsymbol{\xi}) \geq  V^{\balpha^{\mathfrak{a}_{\epsilon}}_S}_{S}(\boldsymbol{\xi}) =\hat{V}^{\bnu^{\bpi^{\mathfrak{a}_{\epsilon}}}}(\mu) \geq \hat{V}(\mu) - \frac{\epsilon}{1 - \beta} \geq V^{\star}(\mu) - \frac{2 \epsilon}{1- \beta}$ where we used the first and second inequalities of the first assertion. Sames computations hold for $V_W(\xi)$ which is enough to end the proof.
\end{proof}

\vspace{0.2 cm}

In order to prove the first point of Theorem \ref{thm: equivalence_value_functions_weak_strong} and according to the verification result in Proposition \ref{proposition : verification theorem}, it  is therefore sufficient to construct for any $\epsilon  >0$,  an  $\epsilon$-optimal lifted randomized feedback policy $\bpi_{\epsilon}$. We notice first that given $\mu \in \Pc_{\lambda}(I \times \Xc)$ and from Proposition \ref{prop : Bellman equation}, we can extract $\mathrm{a}_{\epsilon}^{\mu} \in L^0(I \times \Xc \times [0,1] ; A)$ such that 
\begin{align}
    V^{\star}(\mu) \leq \big[ \T^{\mathrm{a}_{\epsilon}^{\mu}}V^{\star}](\mu) + \epsilon.
\end{align}
The difficulty is to ensure that the map $(\mu,u,x,v) \mapsto \mathfrak{a}_{\epsilon}(\mu,u,x,v) = \mathrm{a}_{\epsilon}^{\mu}(u,x,v)$ is measurable so we can define its associated lifted randomized feedback policy $\bpi^{\mathfrak{a}_{\epsilon}}$ which helds an $\epsilon$-optimal feedback policy for $V^{\star}$.
The main idea of the proof in \cite{motte2022mean} is to rely on a suitable discretization of the space $\Pc_{\lambda}(I \times \Xc)$. In fact, it it known that under the topology of the Wasserstein distance and since $I \times \Xc$ is a compact set, the space $\Pc_{\lambda}(I \times \Xc)$ is compact as a closed subset of the compact set $\Pc(I \times \Xc)$.

\begin{Remark}\textnormal{(Choice of the distance on $\Pc_{\lambda}(I \times \Xc)$).}
\normalfont
   We now briefly explain an important point that motivates our choice of working with the standard Wasserstein distance on $\Pc(I \times \Xc)$, and hence on the induced space $\Pc_{\lambda}(I \times \Xc)$. A seemingly natural alternative would be to endow $\Pc_{\lambda}(I \times \Xc)$ with the following distance. For $\mu, \mu' \in \Pc_{\lambda}(I \times \Xc)$, define
\begin{align}
    \tilde{\Wc}(\mu,\mu') 
    := \int_I \Wc(\mu^u, \mu'^u)\,\lambda(\d u),
\end{align}
where $\Wc$ denotes the usual $1$-Wasserstein distance on $\Pc(\Xc)$, and where $(\mu^u)_{u \in I}$ and $(\mu'^u)_{u \in I}$ denote the $\lambda(\d u)$-a.e. uniquely defined families of disintegrated probability measures.

However, it is not clear that the space $\Pc_{\lambda}(I \times \Xc)$ is compact under the topology induced by the distance $\tilde{\Wc}$. In fact, compactness generally fails unless additional regularity assumptions are imposed on the label variable, more precisely on the map $u \mapsto \mu^u$. Since the space $\Pc_{\lambda}(I \times \Xc)$ is too large in this sense, one could instead restrict attention to a regular subset, denoted by $\Pc_{\lambda}^{\mathrm{reg}}(I \times \Xc)$, consisting of probability measures whose disintegrations satisfy suitable regularity conditions ensuring compactness.

Such a restriction would however require verifying that the associated MDP is well posed on the space $\Pc_{\lambda}^{\mathrm{reg}}(I \times \Xc)$. In particular, for each $t \in \N$, one would need to study the regularity of the mapping
\begin{align}
    u \;\longmapsto\; \P^0_{X_{t+1} \mid U=u}
    = F\big(u,\cdot,\cdot,\P^0_{(U,X_t,\alpha_t)},\cdot,\epsilon_{t+1}^0\big) \sharp
    \big( \P^0_{(X_t,\alpha_t)\mid U=u} \otimes \lambda_{\epsilon} \big).
\end{align}
To ensure that, for all $t \in \N$ and $\P$-a.s., the joint law $\P^0_{(U,X_t)}$ belongs to $\Pc_{\lambda}^{\mathrm{reg}}(I \times \Xc)$, one would have to restrict the class of admissible controls by imposing additional regularity with respect to the label variable. Since we do not wish to introduce such constraints on the control policies, we instead assume Lipschitz regularity of the coefficients $F$ and $f$ with respect to the label argument. This assumption allows us to work on the full space $\Pc_{\lambda}(I \times \Xc)$ endowed with the standard Wasserstein topology.

\end{Remark}

We now show similarly adapting the proof in \cite{motte2022mean} the existence of $\epsilon$-optimal randomized feedback policy for the lifted MDP on $\Pc_{\lambda}(I \times \Xc)$ which will ensure  following Proposition \ref{proposition : verification theorem} that $\hat{V}$ and the optimal values of the CNEMF-MDP either for the weak and the strong formulation are solution to the Bellman fixed point equation.

\vspace{0.2cm}

\noindent \textbf{Proof of Theorem \ref{thm: equivalence_value_functions_weak_strong}.}

\noindent 
(i) From Theorem 4.1 in \cite{motte2022mean} relying on the compacity of the set $\Pc_{\lambda}(I \times \Xc)$ when endowed with its usual Wasserstein distance, we obtain that  for all $\epsilon > 0$, there exists a lifted randomized feedback policy $\bpi^{\mathfrak{a}_{\epsilon}}$, for some $\mathfrak{a}_{\epsilon} \in L^0(\Pc_{\lambda}(I \times \Xc) \times I \times \Xc \times [0,1];A)$ that is $\epsilon$-optimal for $V^{\star}$ according to \eqref{eq : verification result feedback policy}. Consequently, the controls $\balpha_S^{\mathfrak{a}_{\epsilon}} \in \Ac^S$ and $\alpha^{\mathfrak{a}_{\epsilon}}_{W} \in \Ac^W$ are respectively $\frac{\epsilon}{1- \beta}$ for $V_S(\boldsymbol{\xi})$ and $V_W(\xi)$, and we have $V_S(\boldsymbol{\xi})=V_W(\xi)=\hat{V}(\mu)=V^{\star}(\mu)$, which shows the law invariance property and that $V_S$ and $V_W$ are solution to the Bellman equation.

\vspace{0.1 cm}
\noindent (ii) Let us now prove the second statement of Theorem \ref{thm: equivalence_value_functions_weak_strong} where we will construct an optimal randomized feedback control (i.e a lifted randomized feedback policy $\bpi$ associated to $\epsilon=0$ in the Definition in \eqref{eq : verification result feedback policy})  for the CNEMF-MDP control problem. The proof follows  essentially the same lines as in \cite{motte2023quantitative} (Appendix A) so we refer to the references therein with an adaptation to our current verification theorem in \ref{proposition : verification theorem} for the final construction of weak and strong optimal randomized feedback controls.

\subsection{Proof of Theorem \ref{Thm : convergence value functions}}

In order to prove Theorem \ref{Thm : convergence value functions}, we will follow the same steps as discussed in \cite{motte2023quantitative}. To this extend, we refer  to Appendix \ref{appendix : bellman_equation_N_agents} for the construction of $\Tc_N$ the Bellman operator of the $N$-agent.

\subsubsection{Some preliminaries results}

Given $\ba \in A^N$ and $a \in L^0(I \times \Xc \times [0,1] ; A)$, we aim to quantify how close are $\T_N^{\ba}$ (recalling \eqref{eq : operator Bellman N-agent before sup})  and $\T^{a}$ (recalling \eqref{eq : definition Bellman operator before sup}) when $\ba$ are $a$ are close enough in some sense. Since the operators $\T_N^{\ba}$ and $\T^a$  are respectively defined on $L^{\infty}_m(\Xc^N)$ and $L^{\infty}_m \big(\Pc_{\lambda}(I \times \Xc) \big)$ and in order to compare such objects, we need to introduce for any function $W \in L^{\infty}_m(\Pc_{\lambda}(I \times \Xc))$ its unlifted function $\widetilde{W} \in L^{\infty}_m(\Xc^N)$ defined by
\begin{align}\label{eq : unlifted function}
    \widetilde{W}(\bx) :=  W \big(\mu_N^{\lambda} \big[\bu, \bx \big] \big), \quad \forall \bx \in \Xc^N,
\end{align}
where we  recall that the measure $\mu_N^{\lambda}\big[\bu,\bx\big]$ has been introduced in \eqref{eq : Empirical Measures def}.

We now give the following lemma which says how close the operators $\T_N^{\ba}$ and $\T^a$ when $\ba$ and $a$ are \emph{"close enough"}. 
\begin{Lemma}\label{lemma : difference of operators with a}
    There exists some positive constant C such that  for all $\mathrm{a} \in L^0(I \times \Xc \times [0,1];A) $, $\ba \in A^N$,  $\bx \in \Xc^N$ and $\big( \tilde{\xi}_{\bx} =\big((U,\xi_{\bx}), Z \big) \sim \mu_N^{\lambda} \big[\bu,\bx \big] \otimes \Uc([0,1])$, we have 
    \begin{align}
        \big | \widetilde{\T^{\mathrm{a}} V}(\bx) -  \T_N^{\ba} \widetilde{V}(\bx) \big | \leq  C \Big( M_N^{\gamma} + \Wc \big(\P_{(\tilde{\xi}_{\bx},a(\tilde{\xi}_{\bx},Z))}, \mu_N^{\lambda} \big[\bu,  \bx ,\ba \big] \big)^{\gamma}  + (\epsilon_N^F)^{\gamma} + \epsilon_N^f  \Big)
    \end{align}
\end{Lemma}

\begin{proof}
    Let $\ba  = (a^i)_{i \in \llbracket 1 , N \rrbracket} \in A^N$ , $\mathrm{a} \in L^0(I \times \Xc \times [0,1] ; A)$ and $\big((U,\xi_{\bx}), Z \big) \sim \mu_N^{\lambda}[\bu, \bx] \otimes \Uc([0,1])$. We recall here for convenience  the definition of the random measure  $\mu_N^{\lambda} \big[ \bu , \boldsymbol{F}_N[\bx,\ba,\boldsymbol{\epsilon}_1] \big]$ valued in  $\Pc_{\lambda}(I \times \Xc)$ as follows
    \begin{align}
        \mu_N^{\lambda} \big[ \bu , \boldsymbol{F}_N[\bx,\ba, \boldsymbol{\epsilon}_1] \big](\d x ,\d u ) := \Big(\sum_{j=1}^{N} \delta_{I_j^N}(u) \delta_{F_N( \frac{j}{N}, x^j,a^j, \mu_N [\bu, \bx, \ba ], \epsilon_1^j, \epsilon_1^0)}(\d x) \Big)\d u .
    \end{align}
    Recalling the definition of $\T^a$ in  \eqref{eq : definition Bellman operator before sup}, $\T_N^{\ba}$ in \eqref{eq : operator Bellman N-agent before sup} and the unlifted function definition in \eqref{eq : unlifted function}, we have
    \begin{align}\label{eq : operators difference}
           \widetilde{\T^{\mathrm{a}} V}(\bx) -  \T_N^{\ba} \widetilde{V}(\bx) &= \E \Big[ f \big((\tilde{\xi}_{\bx}, \mathrm{a} \big(\tilde{\xi}_{\bx},Z \big),\P_{(\tilde{\xi}_{\bx},\mathrm{a} (\tilde{\xi}_{\bx},Z))} \big)     - \frac{1}{N} \sum_{i=1}^{N}  f_N \big(\frac{i}{N},x^i,a^i, \mu_N \big[ \bu, \bx,\ba \big] \big) \Big] \notag \\
           &\qquad+ \beta \E \Big[ V \big( \P^0_{\tilde{F} \big(\tilde{\xi}_{\bx},\mathrm{a}(\tilde{\xi}_{\bx},Z), \P_{(\tilde{\xi}_{\bx},\mathrm{a}(\tilde{\xi}_{\bx},Z))}, \epsilon_1,\epsilon_1^0 \big)} \big)   -  V \big( \mu_N^{\lambda} \big[\bu , F_N (\bx,\ba , \boldsymbol{\epsilon}_1 )\big] \big) \Big].
    \end{align}
\noindent \textbf{Estimation of the ``$f$" term :}      
\begin{align}
    \E \Big[  f(\tilde{\xi}_{\bx}, \mathrm{a}(\tilde{\xi}_{\bx},Z), \P_{(\tilde{\xi}_{\bx}, \mathrm{a}(\tilde{\xi}_{\bx},Z))} \big) - \frac{1}{N} \sum_{i=1}^{N} f_N \big(\frac{i}{N}, x^i,a^i,\mu_N \big[\bu,\bx,\ba \big] \big) \Big] &= \hat{f}(\P_{(\tilde{\xi}_{\bx}, \mathrm{a}(\tilde{\xi}_{\bx},Z))}) - \hat{f} \big(\mu_N^{\lambda} \big[\bu, \bx, \ba \big] \big) \\
    &\quad + \hat{f}(\mu_N^{\lambda} \big[\bu,\bx,\ba \big] ) - \hat{f}(\mu_N\big[\bu, \bx,\ba] ) \notag \\
    &\quad+ \hat{f}(\mu_N \big[\bu, \bx, \ba \big])- \hat{f}_N(\mu_N \big[\bu , \bx,\ba \big]),
\end{align}
where we denote $ \hat{f}(\mu) := \int_{I \times \Xc \times A} f(u,x,a,\mu)\mu(\d u , \d x , \d a )$ and 
$ \hat{f}_N(\mu) :=  \int_{I \times \Xc \times A} f_N(u,x,a,\mu)\mu(\d u , \d x , \d a)$ for any $\mu \in \Pc(I \times \Xc \times A)$.

\vspace{1mm}

\noindent \textbf{Estimate of $\hat{f}(\P_{(\tilde{\xi}_{\bx}, \mathrm{a}(\tilde{\xi}_{\bx},V))}) - \hat{f}(\mu_N^{\lambda} \big[\bu, \bx,\ba \big])$ :} Using the  definition introduced above, we have 
\begin{align}
    \hat{f}(\P_{(\tilde{\xi}_{\bx}, \mathrm{a}(\tilde{\xi}_{\bx},Z))}) - \hat{f}(\mu_N^{\lambda} \big[\bu, \bx,\ba \big]) &= \int_{I \times \Xc \times A} f \big(u,x,a,\P_{(\tilde{\xi}_{\bx}, \mathrm{a}(\tilde{\xi}_{\bx},Z))}) \Big[\P_{(\tilde{\xi}_{\bx}, \mathrm{a}(\tilde{\xi}_{\bx},Z))} - \mu_N^{\lambda} \big[\bu , \bx, \ba \big] \Big](\d u, \d x , \d a)  \notag \\
    &\quad + \int_{I \times \Xc \times A} \Big[f \big(u,x,a, \P_{(\tilde{\xi}_{\bx}, \mathrm{a}(\tilde{\xi}_{\bx},Z))}\big) - f(u,x,a, \mu_N^{\lambda} \big[\bu , \bx,\ba ] \big) \Big] \mu_N^{\lambda} \big[\bu , \bx,\ba \big](\d u, \d x, \d a) \notag \\
    &= \textbf{I} + \textbf{II}.
\end{align}

\noindent The bound for $\textbf{I}$ follows from Kantorovitch duality result \eqref{eq : Kantorovitch duality} as $\P_{(\tilde{\xi}_{\bx}, \mathrm{a}(\tilde{\xi}_{\bx},Z))}$ and $\mu_N^{\lambda} \big[ \bu, \bx,\ba \big]$ are in $\Pc_{\lambda}(I \times \Xc \times A)$ and the fact that the map $(u,x,a) \mapsto f(u,x,a,\P_{(\tilde{\xi}_{\bx}, \mathrm{a}(\tilde{\xi}_{\bx},Z))})$ is Lipschitz. Therefore, we have
\begin{align}
     \int_{I \times \Xc \times A} f \big(u,x,a,\P_{(\tilde{\xi}_{\bx}, \mathrm{a}(\tilde{\xi}_{\bx},Z))}) \Big[\P_{(\tilde{\xi}_{\bx}, \mathrm{a}(\tilde{\xi}_{\bx},Z))} - \mu_N^{\lambda} \big[\bu , \bx, \ba \big] \Big](\d u, \d x , \d a)  \leq L_f \Wc \big(\P_{(\tilde{\xi}_{\bx}, \mathrm{a}(\tilde{\xi}_{\bx},Z))},\mu_N^{\lambda} \big[\bu , \bx, \ba \big] \big)
\end{align}
The bound for $\textbf{II}$ follows directly from the Lipschitz property in the measure argument of $f$ in Assumption \ref{assumptions: F and f} such that we have $ \textbf{II} \leq  L_f \Wc(  \P_{(\tilde{\xi}_{\bx}, \mathrm{a}(\tilde{\xi}_{\bx},Z))},\mu_N^{\lambda}\big[\bu, \bx,\ba])$.
Therefore, we end up with
\begin{align}
     \Big | \hat{f}(\P_{(\tilde{\xi}_{\bx}, \mathrm{a}(\tilde{\xi}_{\bx},Z))} - \hat{f}(\mu_N^{\lambda} \big[\bu , \bx,\ba \big]) \Big | \leq  2 L_f\Wc(\P_{(\tilde{\xi}_{\bx}, \mathrm{a}(\tilde{\xi}_{\bx},Z))}, \mu_N^{\lambda} \big[\bu, \bx,\ba \big]).
\end{align}

\vspace{0.2 cm}

\noindent \textbf{Estimate of $\hat{f}(\mu_N^{\lambda}\big[ \bu , \bx,\ba \big]) - \hat{f}(\mu_N\big[\bu,\bx,\ba \big])$ :}
 Using the definition of $\mu_N^{\lambda} \big[\bu, \bx,\ba \big]$ and $\mu_N \big[\bu , \bx,\ba \big]$, we have
\begin{align}
    \hat{f}(\mu_N^{\lambda}\big[\bu , \bx,\ba \big]) - \hat{f}(\mu_N \big[\bu,\bx,\ba \big]) &= \sum_{j=1}^{N} \int_{\frac{j-1}{N}}^{\frac{j}{N}} \Big[ f(u,x^j,a^j, \mu_N^{\lambda} \big[\bu , \bx,\ba \big]) - f( \frac{j}{N} , x^j,a^j , \mu_N^{\lambda} \big[\bu , \bx,\ba \big]) \Big] \d u  \notag \\
    &\quad + \frac{1}{N} \sum_{j=1}^{N} \Big[ f(\frac{j}{N},x^j,a^j,\mu_N^{\lambda} \big[\bu ,\bx,\ba \big]) - f(\frac{j}{N}, x^j,a^j , \mu_N \big[\bu , \bx,\ba]) \Big]  \notag \\
    &= \textbf{III} + \textbf{IV}.
\end{align}
\noindent The bound for $\textbf{III}$ follows from the Lipschitz assumption on the label argument of $f$ such that we have $ \textbf{III} \leq  \frac{L_f}{2N}$.
\noindent  The bound for $\textbf{IV}$ follows directly from the Lipschitz property in the measure argument of $f$ in Assumption \ref{assumptions: F and f} and Lemma \ref{lemma : control wasserstein} such that we have $ \textbf{IV} \leq L_f \Wc(\mu_N^{\lambda} \big[\bu ,\bx,\ba \big], \mu_N \big[\bu ,\bx,\ba \big]) \leq  \frac{L_f}{2N}$
Therefore, we end up with
\begin{align}
 \Big | \hat{f}(\mu_N^{\lambda} \big[\bu,\bx,\ba \big])  -\hat{f}(\mu_N \big[\bu , \bx, \ba \big] )  \Big |\leq 
     \frac{L_f}{N}.
\end{align}
\noindent \textbf{Estimate of $\hat{f}(\mu_N\big[\bu, \bx,\ba \big]) - \hat{f}_N(\mu_N \big[\bu ,\bx ,\ba \big])$ :} This follows directly from Assumption \ref{assumption : fN to f and FN to F}
\begin{align}
    \big | \frac{1}{N} \sum_{j=1}^{N} f(\frac{j}{N} , x^j,a^j, \mu_N \big[\bu ,\bx,\ba \big]) - f_N(\frac{j}{N}, x^j,a^j,\mu_N \big[\bu , \bx, \ba \big]) \big | 
    &\leq \epsilon_N^f.
\end{align}
Merging all the terms, we end up with 
\begin{align}\label{eq : control of f term}
   &   \bigg | \E \Big[  f(\tilde{\xi}_{\bx}, \mathrm{a}(\tilde{\xi}_{\bx},V),  \P_{(\tilde{\xi}_{\bx},\mathrm{a}(\tilde{\xi}_{\bx},V))} \big) - \frac{1}{N} \sum_{i=1}^{N} f_N \big(\frac{i}{N}, x^i,a^i,\mu_N \big[\bu , \bx,\ba \big] \big) \Big] \bigg| \\
\leq & \;\;\;    2 L_f \Wc(\tilde{\xi}_{\bx},\mathrm{a}(\tilde{\xi}_{\bx},V), \mu_N^{\lambda} \big[ \bu , \bx ,\ba \big])   + \frac{L_f}{N} + \epsilon_N^f.  
\end{align}

\vspace{1mm}

\noindent \textbf{Estimation of the ``$F$" term : }
Let us  now focus on the second term in \eqref{eq : operators difference}. From the Hölder property of the value function $V$ and the Jensen's inequality since $x \mapsto x^{\gamma}$ is concave (recalling that $\gamma \leq 1$), we have 
\begin{align}\label{eq : difference value functions}
    &\Big |\E \Big[ V \big( \P^0_{\tilde{F}(\tilde{\xi}_{\bx},\mathrm{a}(\tilde{\xi}_{\bx},Z),\P_{(\tilde{\xi}_{\bx},\mathrm{a}(\tilde{\xi}_{\bx},Z))}, \epsilon_1,\epsilon_1^0)} \big)   -  V \big( \mu_N^{\lambda} [\bu , F_N(\bx,\ba, \boldsymbol{\epsilon}_1)] \big) \Big] \Big | \notag \\
    \leq  &\; K_* \E \Big[ \Wc \big(\P^0_{\tilde{F}(\tilde{\xi}_{\bx},\mathrm{a}(\tilde{\xi}_{\bx},Z), \P_{(\tilde{\xi}_{\bx},\mathrm{a}(\tilde{\xi}_{\bx},Z))}, \epsilon_1,\epsilon_1^0)},\mu_N^{\lambda} [\bu , F_N(\bx,\ba,\boldsymbol{\epsilon}_1) \big] \big) \Big]^{\gamma}.
\end{align}
Defining now $\big(\tilde{\xi}^i =(U^i,\xi^i),Z_0^i\big)_{i \in \llbracket 1, N \rrbracket}$ be  $N$ i.i.d random variables of $\boldsymbol{\epsilon}_1 :=\big((\epsilon_1^i)_{i \in \llbracket 1, N \rrbracket}, \epsilon^0 \big)$, such that $(\tilde{\xi}^i,Z_0^i) \sim \mu_N^{\lambda} \big[\bu ,\bx \big] \otimes  \Uc([0,1])$  for any $i \in \llbracket 1, N \rrbracket$. For any i.i.d random variables  $(\tilde{\epsilon}_1^i)_{i \in \llbracket 1, N \rrbracket}$ such that $    (( \tilde{\xi}^i,Z_0^i,\tilde{\epsilon}_1^i)_{i \in \llbracket 1, N \rrbracket}, \epsilon_1^0) \overset{d}{=} (( \tilde{\xi}^i,Z_0^i,\epsilon_1^i)_{i \in \llbracket 1, N \rrbracket}, \epsilon_1^0)$.
Then, we have
\begin{align}\label{eq : bound on control term F}
    &\E \Big[ \Wc \big(\P^0_{\tilde{F}(\tilde{\xi}_{\bx},\mathrm{a}(\tilde{\xi}_{\bx},Z), \P_{(\tilde{\xi}_{\bx},\mathrm{a}(\tilde{\xi}_{\bx},Z))}, \epsilon_1,\epsilon_1^0)},\mu_N^{\lambda} [ \bu , F_N(\bx,\ba,\boldsymbol{\epsilon}_1)] \big) \Big]  \notag \\
    \quad &\leq \E \Big[   \Wc(\P^0_{\tilde{F}(\tilde{\xi}_{\bx},\mathrm{a}(\tilde{\xi}_{\bx},Z),  \P_{(\tilde{\xi}_{\bx},\mathrm{a}(\tilde{\xi}_{\bx},Z))}, \epsilon_1,\epsilon_1^0)},\frac{1}{N} \sum_{i=1}^{N} \delta_{\tilde{F}(\tilde{\xi}^i, \mathrm{a}(\tilde{\xi}^i,Z_0^i), \P_{(\tilde{\xi}_{\bx},\mathrm{a}(\tilde{\xi}_{\bx},Z))}, \tilde{\epsilon}_1^i,\epsilon_1^0)} ) \Big] \notag \\
    &\quad + \E \Big[   \Wc(\frac{1}{N} \sum_{i=1}^{N} \delta_{\tilde{F(}\tilde{\xi}^i, \mathrm{a}(\tilde{\xi}^i,Z_0^i), \P_{(\tilde{\xi}_{\bx},\mathrm{a}(\tilde{\xi}_{\bx},Z))}, \tilde{\epsilon}_1^i,\epsilon_1^0)},  \frac{1}{N} \sum_{i=1}^{N} \delta_{\tilde{F}(\frac{i}{N},x^i,a^i, \mu_N^{\lambda} [\bx,\ba],\epsilon_1^i,\epsilon_0) }) \Big] \notag \\
    &\quad +  \E \Big[ \Wc \big(\frac{1}{N} \sum_{i=1}^{N} \delta_{\tilde{F}(\frac{i}{N},x^i,a^i, \mu_N^{\lambda} [\bx,\ba],\epsilon_1^i,\epsilon_0)},\mu_N^{\lambda} [\bu , F(\bx,\ba,\boldsymbol{\epsilon}_1) ]  \big)\Big]  \notag \\
    &\quad + \E \Big[ \Wc \big(\mu_N^{\lambda} [\bu , F(\bx,\ba,\boldsymbol{\epsilon}_1)], \mu_N^{\lambda} [\bu, F_N (\bx,\ba,\boldsymbol{\epsilon}_1)] \big)     \Big].
\end{align}

\noindent \textbf{Estimate of $\E \Big[   \Wc(\P^0_{\tilde{F}(\tilde{\xi}_{\bx},a(\tilde{\xi}_{\bx},Z),  \P_{(\tilde{\xi}_{\bx},\mathrm{a}(\tilde{\xi}_{\bx},Z))}, \epsilon_1,\epsilon_1^0)},\frac{1}{N} \sum_{i=1}^{N} \delta_{\tilde{F}(\tilde{\xi}^i, \mathrm{a}(\tilde{\xi}^i,Z_0^i), \P_{(\tilde{\xi}_{\bx},\mathrm{a}(\tilde{\xi}_{\bx},Z))}, \tilde{\epsilon}_1^i,\epsilon_1^0)} ) \Big]$ : } 
\noindent This estimate follows directly from the definition of $M_N$ in \eqref{eq : M_N definition} such that we can write
\begin{align}
    \E \Big[   \Wc(\P^0_{\tilde{F}(\tilde{\xi}_{\bx},\mathrm{a}(\tilde{\xi}_{\bx},Z),  \P_{(\tilde{\xi}_{\bx},\mathrm{a}(\tilde{\xi}_{\bx},Z))}, \epsilon_1,\epsilon_1^0)},\frac{1}{N} \sum_{i=1}^{N} \delta_{\tilde{F}(\tilde{\xi}^i, \mathrm{a}(\tilde{\xi}^i,Z_0^i), \P_{(\tilde{\xi}_{\bx},\mathrm{a}(\tilde{\xi}_{\bx},Z))}, \tilde{\epsilon}_1^i,\epsilon_1^0)} ) \Big] \leq M_N.
\end{align}
\vspace{0.1 cm}
\noindent \textbf{Estimate of $\E \Big[   \Wc(\frac{1}{N} \sum_{i=1}^{N} \delta_{\tilde{F}(\tilde{\xi}^i, \mathrm{a}(\tilde{\xi}^i,Z_0^i),\P_{(\tilde{\xi}_{\bx},\mathrm{a}(\tilde{\xi}_{\bx},Z))}, \tilde{\epsilon}_1^i,\epsilon_1^0)},  \frac{1}{N} \sum_{i=1}^{N} \delta_{\tilde{F}(\frac{i}{N},x^i,a^i, \mu_N^{\lambda}[\bu,\bx,\ba],\epsilon_1^i,\epsilon_0) }) \Big]$ :}

\noindent We denote by $\sigma^{ \big(\tilde{\xi}^i,\mathrm{a}(\tilde{\xi}^i,Z_0^i)\big)_{i \in \llbracket 1, N \rrbracket},(\frac{i}{N}, x^i,a^i)_{ i \in \llbracket 1, N \rrbracket}}$ the optimal permutation defined according to the definition in Lemma \ref{lemma : Mesurable Optimal Permutation} that we simplify by denoting $\sigma$. Now, from similar arguments to the ones used in \cite{motte2023quantitative} (Lemma 3.1), we can apply the previous relation to $(\tilde{\epsilon}^i_1)_{i \in \llbracket 1, N \rrbracket} = (\epsilon_1^{(\sigma^{-1})i})_{i \in \llbracket 1, N \rrbracket}$ and after similar computations, we can obtian
\begin{align}
    & \E \Big[   \Wc(\frac{1}{N} \sum_{i=1}^{N} \delta_{\tilde{F}(\tilde{\xi}^i, \mathrm{a}(\tilde{\xi}^i,Z_0^i),\P_{(\tilde{\xi}_{\bx},\mathrm{a}(\tilde{\xi}_{\bx},Z))}, \epsilon_1^{(\sigma^{-1})i},\epsilon_1^0)},  \frac{1}{N} \sum_{i=1}^{N} \delta_{\tilde{F}(\frac{i}{N},x^i,a^i, \mu_N^{\lambda}[\bu,\bx,\ba],\epsilon_1^i,\epsilon_1^0) }) \Big] \\
    \leq& \;  L_{\tilde{F}} \Big( M_N + \Wc(\mu_N^{\lambda} \big[\bu ,\bx,\ba \big], \mu_N\big[\bu , \bx ,\ba \big])  \; + \;  2\Wc \big(\P_{(\tilde{\xi}_{\bx},\mathrm{a}(\tilde{\xi}_{\bx},Z))},\mu_N^{\lambda}\big[\bu,\bx,\ba \big] \Big). 
\end{align}

\vspace{0.1 cm}

\noindent \textbf{Estimate of $\E \Big[ \Wc \big(\frac{1}{N} \sum_{i=1}^{N} \delta_{\tilde{F}(\frac{i}{N},x^i,a^i, \tilde{\mu}_N \big[\bx,\ba],\epsilon_1^i,\epsilon_0)},\mu_N^{\lambda} \big[\bu,F(\bx,\ba, \boldsymbol{\epsilon}_1) \big]  \big)\Big]$ :} This term is handled similarly  following the Lipschitz assumption on the measure argument of $F$ and with similar arguments to the proof of Lemma \ref{lemma : control wasserstein}, we end up with 
\begin{align}
    \E \Big[ \Wc \big(\frac{1}{N} \sum_{i=1}^{N} \delta_{\tilde{F}(\frac{i}{N},x^i,a^i, \mu_N^{\lambda} [\bu,\bx,\ba],\epsilon_1^i,\epsilon_0)},\mu_N^{\lambda} \big[\bu , F(\bx,\ba,\boldsymbol{\epsilon}_1) \big]  \big)\Big] & \leq  \frac{C}{N}.
\end{align}
\noindent \textbf{Estimate of $ \E \Big[\Wc (\mu_N^{\lambda} \big[\bu, F(\bx,\ba,\boldsymbol{\epsilon}_1) \big], \mu_N^{\lambda} \big[\bu , F_N (\bx,\ba,\boldsymbol{\epsilon}_1) \big] \big) \Big]$ :} This follows directly from Assumption \ref{assumption : fN to f and FN to F} since we have
\begin{align}
   &  \E \Big[ \Wc( \mu_N^{\lambda} \big[\bu,F(\bx,\ba,\boldsymbol{\epsilon}_1) \big], \mu_N^{\lambda}\big[\bu,F_N (\bx,\ba,\boldsymbol{\epsilon}_1) \big] \big)     \Big] \\
    \leq & \; \E \Big[ \frac{1}{N} \sum_{i=1}^{N} d   \big(F(\frac{i}{N},  x^i,a^i , \mu_N^{\lambda} \big[\bu,\bx,\ba] , \epsilon_1^i, \epsilon_1^0),F_N(\frac{i}{N},  x^i,a^i , \mu_N^{\lambda} \big[\bu,\bx,\ba] , \epsilon_1^i, \epsilon_1^0) \big)  \Big] \; \leq \;  \epsilon_N^F. 
\end{align}
Recalling \eqref{eq : bound on control term F}, we have 
\begin{align}
&\E \Big[ \Wc(\P^0_{\tilde{F}(\tilde{\xi}_{\bx},\mathrm{a}(\tilde{\xi}_{\bx},Z), \P_{(\tilde{\xi}_{\bx},\mathrm{a}(\tilde{\xi}_{\bx},Z))}, \epsilon_1,\epsilon_1^0)},\mu_N^{\lambda} \big[\bu, F_N(\bx,\ba,\boldsymbol{\epsilon}_1) \big]) \Big] \notag \\
\quad \leq  & \hspace{0.1 cm}M_N + L_{\tilde{F}} \Big(M_N + \Wc(\mu_N^{\lambda}[\bu,\bx,\ba], \mu_N\big[\bu,\bx ,\ba \big]) + 2\Wc \big(\P_{(\tilde{\xi}_{\bx},\mathrm{a}(\tilde{\xi}_{\bx},Z))},\mu_N^{\lambda}\big[\bu,\bx,\ba \big] \Big)  + \frac{C}{N} + \epsilon_N^F ,
\end{align}
which implies by \eqref{eq : difference value functions}, Lemma \eqref{lemma : control wasserstein} and  since $\frac{1}{N} = o(M_N)$ the existence of  a positive constant $C \geq 0$ such that
\begin{align}
    \Big | \E \Big[ V \big( \P^0_{\tilde{F}(\tilde{\xi}_{\bx},\mathrm{a}(\tilde{\xi}_{\bx},Z), \P_{(\tilde{\xi}_{\bx},\mathrm{a}(\tilde{\xi}_{\bx},Z)}, \epsilon_1,\epsilon_1^0)} \big)   -  V \big( \mu_N^{\lambda} \big[\bu, F_N[\bx,\ba,\boldsymbol{\epsilon}_1]\big] \big) \Big] \Big |  \leq  C\Big (M_N^{\gamma} + (\epsilon_N^F)^{\gamma} + \Wc \big(\P_{(\tilde{\xi}_{\bx},\mathrm{a}(\tilde{\xi}_{\bx},Z))},\mu_N^{\lambda}\big[\bu,\bx,\ba \big] \Big)^{\gamma} \Big)
\end{align}

\noindent Together with \eqref{eq : control of f term} and plugging into \eqref{eq : operators difference}, we have 
\begin{align}
    \big | \widetilde{\T^a V}(\bx) -  \T_N^{\ba} \widetilde{V}(\bx) \big | &\leq 2 L_f \Wc(\P_{(\tilde{\xi}_{\bx},\mathrm{a}(\tilde{\xi}_{\bx},Z))}, \mu_N^{\lambda} \big[ \bu , \bx ,\ba \big])  +  \epsilon_N^f  + C\Big (M_N^{\gamma}  + (\epsilon_N^F)^{\gamma} + \Wc \big(\P_{(\tilde{\xi}_{\bx},\mathrm{a}(\tilde{\xi}_{\bx},Z))},\mu_N^{\lambda}\big[\bu,\bx,\ba \big] \Big)^{\gamma} \Big) \notag \\
    &\leq C\Big( M_N^{\gamma} + \Wc(\P_{(\tilde{\xi}_{\bx},\mathrm{a}(\tilde{\xi}_{\bx},Z))}, \mu_N^{\lambda} \big[\bu, \bx ,\ba \big])^{\gamma} +  (\epsilon_N^F)^{\gamma} + \epsilon_N^f  \Big)
\end{align}
where we recall that $\gamma \leq 1$ and using the fact that $\Wc \big(\P_{(\tilde{\xi}_{\bx},a(\tilde{\xi}_{\bx},Z))},\mu_N^{\lambda}[\bu,\bx,\ba] \big)$ is bounded by a constant depending of $\Delta_{I \times \Xc \times A}$. This ends the proof. 
\end{proof}

\noindent The previous Lemma \ref{lemma : difference of operators with a} suggests us for any $\bx \in \Xc^N$ to look for suitable choices of $\mathrm{a} \in L^0(I \times \Xc \times [0,1];A)$ and $\ba \in A^N$ such that the Wasserstein distance $\Wc \big(\P_{(\tilde{\xi}_{\bx},\mathrm{a}(\tilde{\xi}_{\bx},Z))},\mu_N^{\lambda}[\bu,\bx,\ba] \big)$ is as small as possible. We give the following lemmas

\begin{Lemma}\label{lemma : bounded with statistic order}
   For  all $\bx :=(x^1,\ldots,x^N) \in \Xc^N$  and $\boldsymbol{\tilde{\xi}} := \big( \tilde{\xi}^i = (U^i,\xi^i)\big)_{i \in \llbracket 1 , N \rrbracket}$ i.i.d with distribution $\mu_N^{\lambda} \big[\bu , \bx \big]$, the following holds true
    \begin{align}
        \E \Big[ \Wc \big(\frac{1}{N} \sum_{i=1}^{N} \delta_{(U^i,\xi^i)} , \frac{1}{N} \sum_{i=1}^{N} \delta_{( \frac{i}{N} , x^i)} \big) \Big] \leq  M_N + \frac{1}{2N}.
    \end{align}
\end{Lemma}

\begin{proof}
    Let $i \in \llbracket 1 , N \rrbracket$. Since $(U^i,\tilde{\xi}^i) \sim \mu_N^{\lambda} [\bu,\bx]$, we have $U^i \sim \Uc([0,1])$ and $\tilde{\xi}^i \sim \frac{1}{N} \sum_{i=1}^{N} \delta_{x^i}$. We use the triangular inequality of the Wasserstein distance to state that 
    \begin{align}
        \E \Big[ \Wc \big(\frac{1}{N} \sum_{i=1}^{N} \delta_{(U^i,\tilde{\xi}^i)} , \frac{1}{N} \sum_{i=1}^{N} \delta_{( \frac{i}{N} , x^i)} \big) \Big] &\leq  \E \Big[ \Wc \big(\frac{1}{N} \sum_{i=1}^{N} \delta_{(U^i,\tilde{\xi}^i)} , \mu_N^{\lambda}[\bu,\bx  ])\Big] + \Wc(\mu_N^{\lambda}[\bu,\bx], \mu_N[\bu,\bx] \big)  \\
        &\leq M_N + \frac{1}{2N}, 
    \end{align}
    where we used the definition of $M_N$ in \eqref{eq : M_N definition} and Lemma \ref{lemma : control wasserstein}.
\end{proof}

\begin{Lemma}\label{lemma : optimal a^a from given law}
    Let $x \in \Xc^N$. Then, for any $\mathrm{a} \in L^0(I \times \Xc \times [0,1];A)$, there exists $\ba^{\mathrm{a}} \in A^N$ such that 
    \begin{align}
        \Wc \big(\P_{(\tilde{\xi}_{\bx}, \mathrm{a}(\tilde{\xi}_{\bx},Z))}, \mu_N^{\lambda} \big[ \bu , \bx ,\ba^{\mathrm{a}}] \big) \leq  2 M_N  + \frac{1}{N},
    \end{align}
    where $(\tilde{\xi}_{\bx},Z) \sim \mu_N^{\lambda} \big[\bu ,\bx \big] \otimes \Uc([0,1])$. Conversely, for any $\ba \in A^N$, there exists $\mathrm{a}^{\ba} \in L^0(I \times \Xc \times [0,1] ;A)$ such that
    \begin{align}
        \Wc\big(\P_{(\tilde{\xi}_{\bx}, \mathrm{a}^{\ba}(\tilde{\xi}_{\bx},Z))}, \mu_N^{\lambda}\big[\bu , \bx ,\ba] \big)=0.
    \end{align}
\end{Lemma}
\begin{proof}
    Fix $\mathrm{a} \in L^0(I \times \Xc \times [0,1];A)$. Let's consider $\boldsymbol{\tilde{\xi}}:= \big(\tilde{\xi}^i =(U^i,\xi^i)\big)_{i \in \llbracket 1,N \rrbracket}$ i.i.d with distribution $\mu_N^{\lambda} \big[ \bu , \bx \big]$, independant from $\boldsymbol{Z}_0 = (Z_0^i)_{i \in \llbracket 1 , N \rrbracket}$, i.i.d $\sim \Uc([0,1])$. Then, we have by denoting $\sigma^{\tilde{\boldsymbol{\xi}},[\bu,\bx]}$ where $[\bu , \bx ]= \big( \frac{1}{N}, x^1), \ldots, (1, x^N) \big)$ the optimal permutation as defined in Lemma \ref{lemma : Mesurable Optimal Permutation}
    \begin{align}  
    &\E \Big[ \Wc\big(\P_{(\tilde{\xi}_{\bx}, \mathrm{a}(\tilde{\xi}_{\bx},Z))}, \frac{1}{N} \sum_{i=1}^{N} \delta_{(\frac{i}{N},x^i,\mathrm{a}(\boldsymbol{\xi}^{\sigma_i^{\boldsymbol{\xi}, [\bu,\bx]}},Z_0^i)} \big)  \Big ] \notag \\
    \leq &\E 
        \Big[  \Wc \big(\P_{(\tilde{\xi}_{\bx}, \mathrm{a}(\tilde{\xi}_{\bx},Z))}, \frac{1}{N} \sum_{i=1}^{N} \delta_{(\boldsymbol{\xi}^{\sigma_i^{\boldsymbol{\xi},[\bu , \bx]}},\mathrm{a}(\boldsymbol{\xi}^{\sigma_i^{\boldsymbol{\xi}, [\bu,\bx]}},Z_0^i)} \big) \Big  ] + \E \Big[ \Wc(\frac{1}{N} \sum_{i=1}^{N} \delta_{(\boldsymbol{\xi}^{\sigma_i^{\boldsymbol{\xi},[\bu , \bx]}},\mathrm{a}(\boldsymbol{\xi}^{\sigma_i^{\boldsymbol{\xi}, [\bu,\bx]}},Z_0^i)} \big),\frac{1}{N} \sum_{i=1}^{N} \delta_{(\frac{i}{N},x^i,\mathrm{a}(\boldsymbol{\xi}^{\sigma_i^{\boldsymbol{\xi}, [\bu,\bx]}},Z_0^i)} \big) \Big] \notag \\
    \leq  &M_N + \E \Big[ \frac{1}{N}  \sum_{i=1}^{N} d_{I\times \Xc} \big( \boldsymbol{\xi}^{\sigma_i^{\boldsymbol{\xi}, [\bu, \bx]}},(\frac{i}{N},x^i)  \big) \Big] ,  \notag \\
    \leq &2 M_N + \frac{1}{2N} 
    \end{align}
    where the last inequality holds from definition of $M_N$, Lemma \ref{lemma : bounded with statistic order} and Lemma \ref{lemma : Mesurable Optimal Permutation}. Recalling Lemma \ref{lemma : control wasserstein} and using Wasserstein's triangular inequality, it follows that  $\P \Big(\Wc \big(\P_{(\tilde{\xi}_{\bx}, \mathrm{a}(\tilde{\xi}_{\bx},Z))},\mu_N^{\lambda} \big[ \bu ,\bx,\mathrm{a}(\boldsymbol{\xi}^{\sigma^{\boldsymbol{\xi}, [\bu,\bx]}},\boldsymbol{Z}_0 \big) \big]\leq  2M_N + \frac{1}{N} \Big) > 0$  which ensures the existence of a vector $\ba^{\mathrm{a}} \in A^N$ such that 
    \begin{align}
        \Wc \big(\P_{(\tilde{\xi}_{\bx}, \mathrm{a}(\tilde{\xi}_{\bx},Z))}, \mu_N^{\lambda} \big[ \bu , \bx,\ba^{\mathrm{a}} \big] \big) \leq 2 M_N+ \frac{1}{N}.
    \end{align}
    On the other hand, given $\ba \in A^N$ and by considering $(\tilde{\xi},\tilde{\alpha}) \sim \mu_N^{\lambda} [\bu , \bx,\ba]$, it is sufficient to choose a kernel $\ba^{\mathrm{a}} \in L^0(I \times \Xc \times [0,1];A)$ for simulating the conditional law of $\tilde{\alpha}$ knowing $\tilde{\xi}$ following \cite{kallenberg2002foundations}.
\end{proof}

\subsubsection{Proof of Theorem \ref{Thm : convergence value functions}}

\begin{proof}

\noindent This proof follows the same line as in the proof of Theorem 2.1 in \cite{motte2023quantitative} once we derived the estimates of Lemma \ref{lemma : difference of operators with a} and \ref{lemma : optimal a^a from given law} (recalling that $\frac{1}{N} =o(M_N)$) so we refer to the references therein.

\end{proof}

\appendix

\renewcommand{\thesection}{\Alph{section}}

\section{Some technical results}

\begin{Lemma}\label{eq : measure products from desintegration}\textnormal{(Measure product from disintegration and maps).}
 Let $F : I \times \Xc \to \Xc$  be a measurable map and $\mu \in L^0(I, \Pc(\Xc))$. Consider now the probability kernel $\nu$ on $I \times \Xc$  defined by $\nu(u)(\d x) :=  F(u,\cdot) \sharp \mu^u \in \Pc(\Xc)$.
Define the measurable mapping $\tilde{F} : I \times \Xc \to I \times \Xc$ as $\tilde{F}(u,x) := \big(u, F(u,x)\big)$. Then the following relation holds 
    \begin{align}\label{eq : equality measures}
        \lambda  \nu = \tilde{F} \sharp\big( \lambda(\d u )\mu^u(\d x) \big).
    \end{align}
\end{Lemma}

\begin{proof}
    Let $\Phi \in L^{\infty}_m(I\times \Xc)$. Defining the measure $\tilde{\mu}(\d u, \d x) := \lambda(\d u)\mu^u(\d x) \in \Pc_{\lambda}(I \times \Xc) $ and starting from the right-hand side of \eqref{eq : equality measures}, we have $\int_{I \times \Xc} \Phi(u,x) \d(\tilde{F} \sharp \tilde{\mu})(u,x) =  \int_{I \times \Xc} \Phi(\tilde{F}(u,x)) \d \tilde{\mu}(u,x) = \int_{I \times \Xc} \Phi(\tilde{F}(u,x))\mu^u(\d x) \lambda(\d u)$. Now from the left-hand side, we have 
    \begin{align}
        \int_{I \times \Xc} \Phi(u,x) \nu(u,\d x) \lambda(\d u) = \int_{I} (\int_{\Xc} \Phi(u,F(u,x) \mu^u(\d x) \big) \lambda(\d u) = \int_{I \times \Xc} \Phi(\tilde{F}(u,x)) \mu^u(dx) \lambda(\d u).
    \end{align}
\end{proof}

\begin{Lemma}\textnormal{(Error approximation between empirical lifted measure and empirical measure).}\label{lemma : control wasserstein}
    For any $\bx \in \Xc^N$, we have
  \begin{align}\label{eq : wasserstein control}
    \Wc(\big(\mu_N^{\lambda} \big[\bu, \bx  \big], \mu_N \big[\bu, \bx \big] \big) \leq \frac{1}{2N}.
\end{align} 
\end{Lemma}
\begin{proof}
    On $\Pc(I \times \Xc)^2$, we define the coupling $\gamma$ as $\gamma(\d x, \d u , \d y , \d v ) := \Big ( \sum_{i=1}^{N}  \mathds{1}_{I_i^N}(u) \ \delta_{\frac{i}{N}}(\d v) \otimes \delta_{x^i }(\d x) \otimes \delta_{x^i}(\d y)  \Big) \d u$.
It is straightforward to verify  that $\gamma$ verifies the marginal conditions and it follows after defining $\tilde{\gamma}(\d u ,\d v) := \sum_{i=1}^{N} \mathds{1}_{I_i^N}(u) \delta_{\frac{i}{N}}(\d v)\d u  $ that 
\begin{align}
    \Wc(\big(\mu_N^{\lambda} \big[ \bu , \bx \big], \mu_N \big[ \bu , \bx \big] \big) \leq \int_{(I \times \Xc)^2} \Big[| u -v|+ d(x,y) \Big] \gamma (\d x, \d u , \d y , \d v) =  \int_{I^2} | u -v | \tilde{\gamma}(\d u ,\d v) = \frac{1}{2N}.
\end{align}
\end{proof}

\begin{Lemma}\label{lemma : Mesurable Optimal Permutation}\textnormal{(Measurable optimal permutation).}
 Let $(\Xc,d)$ be a metric space. There exists a measurable map $ \sigma : (\Xc^N)^2 \ni (\bx,\bx') \mapsto \sigma^{\bx,\bx'}  \in \Sc_{N}$ where $\Sc_{N}$ denotes all the permutations of $\llbracket 1, N \rrbracket$ such that for all $\bx :=(x^1,x^2,\ldots,x^N)$ and $\bx':= (x'^1,x'^2,\ldots,x'^N)$, we have
\begin{align}\label{eq : optimal coupling empirical measures}
    \Wc \big(\frac{1}{N} \sum_{i=1}^{N}  \delta_{x^i} , \frac{1}{N} \sum_{i=1}^{N} \delta_{x'^i} \big) = \frac{1}{N} \sum_{i=1}^{N} d(x^i,x'^{\sigma_i^{\bx,\bx'}}).
\end{align}
where $\sigma_i^{\bx,\bx'}$ denotes the image of $i$ by the permutation $\sigma^{\bx,\bx'}$.
\end{Lemma}

\begin{proof}
    The proof of this known result can be found in \cite{thorpe2018introduction}.
\end{proof}

\section{Lifted MDP on $\Pc_{\lambda}(I \times \Xc)$ in the strong formulation}\label{subsec: lifted_strong_formulation}

We  denote by $\lbrace \mu_t = \P^0_{X_t^u}(\d x) \d u, t \in \N \rbrace$ the random $\Pc_{\lambda}(I \times \Xc)$-valued process, which is $\F^0$-adapted by Proposition A.1 in \cite{motte2022mean}. We consider the $\F^0$-adapted control  process to be $\balpha_t = \P^0_{(X_t^v,\alpha_t^v)}(\d x , \d a) \d v$ valued in the space of probability measures $\Pc(I \times \Xc \times A)$, endowed with its natural $\sigma$-algebra with the Wasserstein metric. We note that this this $\boldsymbol{A}$-valued process has to satisfy the property that $\pr_{12} \sharp \balpha_t = \mu_t$ at any $t \in \N$.
\begin{Lemma}
Let $\tilde{F}$ the map introduced in \eqref{eq : definition tilde F}. Then, we have
\begin{align}\label{eq : dynamics of mu_strong_formulation}
    \mu_{t+1} =  \tilde{F} \big(\cdot,\cdot,\cdot, \P^0_{(X_t^v,\alpha_t^v)}(\d x , \d a)\d v , \cdot ,\epsilon_{t+1}^0 \big) \sharp \big( \P^0_{(X_t^v,\alpha_t^v)}(\d x , \d a)\d v \otimes \lambda_{\epsilon} \big), \quad t \in \N, \quad \text{$\P$-a.s}, .
\end{align}
\end{Lemma}
\begin{proof}
    Starting from the dynamics of $X_t^u$ in \eqref{eq:dynamics_NEMFC}, we denote $\P-\text{a.s}$ by $\nu=(\nu_t)_{t \in \N^*} \in L^0(I; \Pc(\Xc)\big)$  the process defined as 
    \begin{align}
        \nu_{t+1}(u) = F \big(u,\cdot,\cdot,\P^0_{(X_t^v,\alpha_t^v)}(\d x , \d a)\d v, \cdot ,\epsilon_{t+1}^0 \big) \sharp \big( \P^0_{(X_t^u,\alpha_t^u)} \otimes \lambda_{\epsilon} \big) \in \Pc(\Xc), \quad t \in \N.
    \end{align}
    Then, we clearly have $\mu_{t+1}= \lambda  \nu_{t+1} \hspace{0.2 cm} \P \text{ a.s}$. Now, it remains to show that 
    \begin{align}
        \tilde{F} \big(\cdot,\cdot,\cdot, \P^0_{(X_t^v,\alpha_t^v)}(\d x , \d a)\d v , \cdot ,\epsilon_{t+1}^0 \big) \sharp \big( \P^0_{(X_t^v,\alpha_t^v)}(\d x , \d a)\d v \otimes \lambda_{\epsilon} \big) =  \lambda  \nu_{t+1}, 
        \quad \P \text{ a .s}, 
\end{align}
and this follows from a straight application of the result in Lemma \ref{eq : measure products from desintegration} after noticing that $(\mu^u(\d x) \otimes \lambda_{\epsilon}) \d u = (\mu^u(\d x) \d  \mu ) \otimes \lambda_{\epsilon}$ from Fubini-Tonelli's theorem. 
\end{proof}
By using the coupling projection $\bp$ from \eqref{eq : coupling_projection} and definition of $\hat{F}$ in \eqref{eq : definition hat F}, the dynamics \eqref{eq : dynamics of mu_strong_formulation} can be rewritten as 
\begin{align}
    \mu_{t+1} = \hat{F} \big(\mu_t, \balpha_t,\epsilon_{t+1}^0 \big) \quad \P-\text{a .s}, \quad t \in \N,
\end{align}
Now, it turns out that the strong formulation leads to the same MDP as the weak formulation with characteristics $\big(\Pc_{\lambda}(I \times \Xc), \bA = \Pc(I \times \Xc \times A), \hat{F},\hat{f}, \beta \big)$ with controlled dynamics given by \eqref{eq : dynamics canonical formulation}-\eqref{eq : value function canonical formulation}. Given $\boldsymbol{\xi}=(\xi^u)_{u \in I}$ admissible according to Definition \ref{assumption : Admissible conditions for the strong formulation}, we can then introduce the relaxed operator $\Lc_{\boldsymbol{\xi}}^0$ as follows
  \begin{align}\label{eq : Lxi mapping_strong_formulation}
  \begin{cases}
  \mathcal{L}_{\boldsymbol{\xi}}^0 : \mathcal{A}^S &\longrightarrow \bAc,  \\ \quad
\balpha = (\alpha_t^u)_{u,t} &\longmapsto \balpha^{R} = (\balpha_t^R)_t 
: \balpha_t^R:= \P^0_{(X_t^v,\alpha_t^v)}(\d x , \d a) \d v , \quad 
t \in  \N .
  \end{cases}
\end{align}

\section{Bellman equation for the $N$-agent MDP in the non exchangeable setting}\label{appendix : bellman_equation_N_agents}

In our non-exchangeable setting, the dynamic relation \eqref{eq : dynamic N agents} can be rewritten as  $\bX_{t+1} = \boldsymbol{F}_N(\bX_t, \balpha_t, \boldsymbol{\epsilon}_{t+1})$ where the measurable state transition map $\boldsymbol{F}_N : \Xc^N \times A^N \times (E^N \times E^0) \to \Xc^N$ is given for  any $\bx=(x^i)_{i \in \llbracket 1, N \rrbracket}$, $\ba = (a^i)_{i \in \llbracket 1, N \rrbracket}$ and  $ \be= \big((e^i)_{i \in \llbracket 1, N \rrbracket},e^0)$ by $    \bF_N(\bx,\ba,\be) := \Big( F_N(\frac{i}{N}, x^i,a^i, \mu_N \big[\bu , \bx,\ba \big],e^i,e^0) \Big)_{i \in \llbracket 1, N \rrbracket}$.
The reward  function $\boldsymbol{f}_N : \Xc^N \times A^N \to \R$ is given by $ \boldsymbol{f}_N(\bx, \ba) := \frac{1}{N} \sum_{i=1}^{N} f_N \big(\frac{i}{N}, x^i, a^i, \mu_N \big[ \bu, \bx , \ba \big] \big)$.
Given $\balpha \in \bAc_N$, the expected gain \eqref{eq : Expected Gain N-agent} can be rewritten as 
\begin{align}
    V_N^{\balpha} (\bx_0) =\E \Big[ \sum_{t \in \N} \beta^t \boldsymbol{f}_N(\bX_t, \balpha_t) \Big].
\end{align}
The MDP with characteristics $(\Xc^N,A^N, \bF_N, \boldsymbol{f}_N, \beta)$ is then well posed. We then  introduce the Bellman operator  $\Tc_N : L^{\infty}_m(\Xc^N) \to L^{\infty}(\Xc^N)$ defined for any $W \in L_m^{\infty}(\Xc^N)$ by 
\begin{align}\label{eq : Bellman_Operator_N_agents}
     \big[ \Tc_N W \big](\bx) := \underset{ \ba \in A^N}{\sup} \T_N^a W(\bx), \quad \bx \in \Xc^N.
\end{align}
where
\begin{align}\label{eq : operator Bellman N-agent before sup}
    \T_N^{\ba}W(\bx):= \boldsymbol{f}_N(\bx, \ba) + \beta \E \Big[W \big(\bF_N(\bx,\ba,\boldsymbol{\epsilon}_1)\big) \Big],\quad \bx \in \Xc^N, \quad \ba \in A^N.
\end{align}
Following closely the arguments in Appendix B in \cite{motte2023quantitative}, we can check that $\Tc_N$ admits a unique fixed point denoted by $V_N^{\star}$ which is $\gamma$-Hölder.  The objective is now to show that $V_N = V_N^{\star}$, i.e, the value function of the $N$-agent MDP is the unique fixed point of the Bellman operator $\Tc_N$. The proof of this standard result relies on the construction of $\epsilon$-optimal feedback policy and on a suitable verification result. We first recall the notion of (randomized) feedback policy in the context of the $N$-agent MDP.

\begin{Definition}\label{def : feedback_policy}
    A feedback policy (resp randomized feedback policy) for the $N$-agent non exchangeable mean field MDP is a mapping $\bpi \in L^0(\Xc^N;A^N)$ (resp $\bpi \in L^0( \Xc^N \times [0,1]^N ; A^N)$. The associated feedback control is the unique control 
\begin{align}\label{eq : feedback_policy_N_agent_MDP}
    \balpha_t^{\bpi} = \bpi(\bX_t), \quad \text{resp}, \quad \balpha_t^{\bpi} = \bpi_{r}(\bX_t,\boldsymbol{Z}_t)),
\end{align}
 where $ \big \lbrace \boldsymbol{Z}_t =(Z_t^i)_{i \in \llbracket 1, N \rrbracket}, t \in \N \big \rbrace$ is  a family of mutually i.i.d uniform random variables  on $[0,1]$ independant of $\boldsymbol{\epsilon}$.
\end{Definition}
Given $\bpi \in L^0(\Xc^N ;A^N)$ (resp $L^0(\Xc^N \times [0,1]^N ; A^N)$), we introduce the operator $\Tc_N^{\bpi}$ on $L_m^{\infty}(\Xc^N)$ defined as 
\begin{align}\label{eq : Tn operator for pi}
    \Tc_N^{\bpi }W(\bx) :=  \boldsymbol{f}_N(\bx,\pi(\bx)) + \beta \E \Big[ W \big(\boldsymbol{F}_N(\bx,\bpi(\bx),\epsilon_1) \big) \Big], \quad \bx \in \Xc^N,
\end{align}
resp
\begin{align}
    \Tc_N^{\bpi} W(\bx) :=  \E \Big[ \boldsymbol{f}_N \big(\bx,\pi(\bx, \boldsymbol{Z}_0) \big) + \beta  W \big(\boldsymbol{F}_N(\bx,\bpi(\bx, \boldsymbol{Z}_0),\epsilon_1)  \Big], \quad \bx \in \Xc^N,
\end{align}
where $\boldsymbol{Z}_0 :=(Z_0^i)_{i \in \llbracket 1 , N \rrbracket}$ is a family of i.i.d uniform random variables independant of $\boldsymbol{\epsilon}$.
\begin{Definition}
    Fix $\epsilon \geq 0$. We say that $\bpi^{\epsilon}$ is an $\epsilon$-optimal (randomized) feedback policy for $V_N^{\star}$ if for any $\bx \in \Xc^N$
    \begin{align}
        V_N^{\star}(\bx) \leq \Tc_N^{\bpi^{\epsilon}} V_N^{\star}(\bx) + \epsilon.
    \end{align}
\end{Definition}
We recall from the verification result (Lemma B.5 in \cite{motte2023quantitative}) that if $\bpi^{\epsilon}$ is an $\epsilon$-optimal (randomized) feedback policy for $V_N^{\star}$, then the associated feedback control $\balpha^{\bpi^{\epsilon}} = (\balpha_t^{\bpi^{\epsilon}})_{t \in \N}$ from Definition \ref{def : feedback_policy}  is $\frac{\epsilon}{1- \beta}$ optimal for $V_N$, i.e satisfies $V_N^{\balpha^{\bpi^{\epsilon}}} \geq V_N - \frac{\epsilon}{1 - \beta}$ and $V_N  \geq V_N^{\star} - \frac{\epsilon}{1 - \beta}$.
\begin{Theorem}
    For all $\epsilon > 0$, there exists a (randomized) feedback policy $\bpi^{\epsilon}$ that is $\epsilon$-optimal for $V_N^{\star}$. Consequently, the control $\balpha^{\bpi^{\epsilon}} \in \bAc$ is $\frac{\epsilon}{1- \beta}$ for $V_N$, and we have $V_N = V_N^*$, which thus satisfies the Bellman fixed point equation with operator $\Tc_N$ defined in \eqref{eq : Bellman_Operator_N_agents}.
\end{Theorem}

\begin{proof}
    The proof of this result can be found in Proposition B.2 in \cite{motte2023quantitative}.
\end{proof}

\begin{small}
\end{small}

\begin{small}
\bibliographystyle{plain}   
\bibliography{References.bib}   
\end{small}

\end{document}